\documentclass{amsart}[12pt]
\usepackage{hyperref}
\usepackage[backrefs]{amsrefs}
\usepackage{float}
\usepackage{amsmath, amssymb, amsfonts, amstext, verbatim, amsthm, mathrsfs}
\usepackage{tikz}

\theoremstyle{plain}
\newtheorem{theorem}{Theorem}[section]
\newtheorem*{theorem*}{Theorem}
\newtheorem*{maintheorem*}{Main Theorem}
\newtheorem{lemma}[theorem]{Lemma}
\newtheorem{proposition}[theorem]{Proposition}

\newtheorem{corollary}[theorem]{Corollary}

\theoremstyle{definition}
\newtheorem{definition}[theorem]{Definition}
\newtheorem{remark}[theorem]{Remark}
\newtheorem{example}[theorem]{Example}

\numberwithin{equation}{section}

\newcommand{\GAP}{\texttt{GAP}}
\newcommand{\gfan}{\texttt{gfan}}
\newcommand{\Macaulay}{\texttt{Macaulay2}}
\newcommand{\MATLAB}{\texttt{MATLAB}}
\newcommand{\Octave}{\texttt{Octave}}\newcommand{\QEPCAD}{\texttt{QEPCAD}}
\newcommand{\SageMath}{\texttt{SageMath}}
\newcommand{\face}{\operatorname{face}}
\newcommand{\cond}{\operatorname{cond}}
\newcommand{\Chow}{\operatorname{Chow}}
\newcommand{\Gr}{\operatorname{Gr}}
\newcommand{\GL}{\operatorname{GL}}

\newcommand{\Sym}{\operatorname{Sym}}
\newcommand{\Span}{\operatorname{Span}}

\newcommand{\lch}{\operatorname{lch}}
\newcommand{\NN}{\mathbb{N}}
\newcommand{\RR}{\mathbb{R}}
\newcommand{\del}{\delta}
\newcommand{\Del}{D}
\newcommand{\condIndex}{\operatorname{c.i.}}
\newcommand{\simp}{\operatorname{simp}}
\newcommand{\prox}{\operatorname{prox}}
\newcommand{\fieldk}{\mathbf{k}}
\newcommand{\aug}{\operatorname{aug}}
\newcommand{\red}{\operatorname{red}}
\newcommand{\Chi}{\mathrm{X}}

\definecolor{green}{rgb}{0,0.5,0}

\theoremstyle{definition}
\numberwithin{equation}{section}

\newcommand{\PP}{\mathbb{P}}

\begin{document}

\title[The worst 1-PS for toric rational
  curves with one unibranch singularity]{The worst destabilizing 1-parameter subgroup for toric rational
  curves with one unibranch singularity}


\author{Joshua Jackson}
\author{David Swinarski}


\date{}

\begin{abstract}
Kempf proved that when a point is unstable in the sense of Geometric
Invariant Theory, there is a ``worst'' destabilizing 1-parameter
subgroup $\lambda^{*}$.  It is natural to ask: what are the worst 1-PS for the unstable points in the GIT
problems used to construct the moduli space of curves
$\overline{M}_g$? Here we consider Chow points of toric rational curves with one
unibranch singular point. We translate the problem as an
explicit problem in convex geometry (finding the closest point on a
polyhedral cone to a point outside it). We prove that the worst 1-PS has a combinatorial description that
persists once the embedding dimension is sufficiently large, and present some
examples.
\end{abstract}

\maketitle

\section{Introduction}

There are several GIT results concerning the Hilbert or Chow stability of
embedded singular curves.  When the dimension of the linear system is sufficiently small compared to the
arithmetic genus, many singularities are GIT stable. For example, cusps are
semistable in several GIT problems, including plane quartics and 2-, 3-, or
4-canonical curves \cites{MFK, HassettHyeon, Schubert, HyeonMorrison, BFMV}. In
contrast, when the dimension of the linear system is sufficiently large compared to the genus,
the only singularities that are GIT semistable are nodes; see \cite{Gieseker} for asymptotic Hilbert semistability or \cite{Mumford} for Chow semistability.

Kempf showed in \cite{Kempf} that when $x\in X$ is GIT unstable, there
is a ``worst'' destabilizing 1-parameter subgroup $\lambda^{*}$ associated to
$x$. It is the worst in the sense that it maximizes
$\mu(x,\lambda)/\|\lambda\|$, where $\mu(x,\lambda)$ is the
Hilbert-Mumford function. Thus, it is natural to ask: when a point
$[C]$ parametrising a singular
curve is GIT unstable, what is the worst 1-PS $\lambda^{*}$?


Knowledge of worst 1-PS has important applications in moduli
theory. Hesselink and Kempf-Ness describe a 
locally closed stratification that is invariant under the group action
and arises from indexing each point by its worst 1-PS; Kirwan uses this stratification to compute the cohomology of the GIT quotient \cites{Hesselink, Ness, Kirwan}. More recently, this whole picture has been greatly generalised by Halpern-Leistern's
Beyond GIT program, in which the HN filtration of an unstable point is a
generalization of the worst 1-PS \cite{Halpern-Leistner}. On the moduli side, it has been shown that these unstable Hesselink-Kempf-Kirwan-Ness strata can themselves be quotiented using results from Non-Reductive GIT \cite{HoskinsJackson}, allowing one to construct moduli spaces of \emph{unstable} objects. Although we perform no such quotients in this paper, a principal motivation for our work here is the construction of new moduli spaces of unstable (i.e.\ singular) curves.

We study the problem of Chow stability for a toric rational curve $C$ with one unibranch
singularity. There are two reasons we begin our study with these
examples. The first reason is that we know where to look: when $X$
has an automorphism group whose action is multiplicity-free, the
Kempf-Morrison Lemma \cite{MS}*{Prop.~4.7} guarantees that the worst 1-PS will appear in the
maximal torus diagonalizing this action. The second reason is that when $X$ is toric, the vertices of the Chow polytope
can be identified with coherent triangulations of the polytope of $X$. This allows us to interpret the value
of the Hilbert-Mumford function $\mu([X],\lambda)$ as the volume of a convex
region. We can go even further when $X$ is a rational curve with one
unibranch singularity and restate the
problem of finding the worst 1-PS as finding the closest point on a polyhedral cone $W$ to a vector $a$ that lies
outside $W$. The cone $W$ and vector $a$ are completely explicit once
the singularity and embedding dimension are given.

Using these ideas, we wrote software that uses \GAP, \Macaulay, and \MATLAB~or
\Octave~\cites{GAP, Macaulay, MATLAB, Octave} to compute many
examples. We observed that, for a fixed singularity, the worst 1-PS
behaves in a predictable way once the embedding dimension $N$ is
sufficiently large. There exists an integer $\ell$ such that the first
$\ell$ weights in the worst 1-PS are
constant with respect to $N$, and the remaining weights are given by
an explicit formula that arises from a least squares linear regression
calculation. We call this phenomenon the \emph{persistence of the
  worst 1-PS for all $N$
sufficiently large}.

We state this more precisely as follows. Let $C$ be a rational curve with one unibranch singularity at
  $p$. Let $\Gamma = \{\gamma_0,\gamma_1,\ldots\}$ be the semigroup of the singularity.  For all $N$, consider the map $C \rightarrow \mathbb{P}^N$ given by 
\[
t \mapsto (1, t^{\gamma_1}, t^{\gamma_2}, \ldots, t^{\gamma_N})  
\]
Let $C_{\Gamma}$ be the image of $C$ under this map, and let $w^{*}$
be the weights of the worst 1-PS.

\begin{maintheorem*}[Persistence of the worst 1-PS]
  There exist an integer $N_0$ and an integer $\ell$ (both depending
  on $\Gamma$) such that for all $N \geq N_0$, the coordinates $w_0^*,\ldots,w_{\ell-1}^{*}$ are
constant with respect to $N$, and the coordinates
$w_{\ell}^{*},\ldots,w_{N}^{*}$ are given by the formula 
\[
w_{i}^{*} = mi + b
\]
where
\begin{align*}
m &= -\frac{6}{(N-\ell+1)(N-\ell+2)} \\
b &= 2+\frac{2(N+2\ell-1)}{(N-\ell+1)(N-\ell+2)}
\end{align*}
\end{maintheorem*}

To prove this theorem, we first prove a similar statement for a
closely related optimisation problem that we call the Simplified
Problem. For each face $F$ of $W$, we use the Karush-Kuhn-Tucker (KKT) 
conditions in nonlinear optimisation to study the proximum
$\operatorname{Prox}(a,\operatorname{Span}(F))$, and show that the
face containing the global optimum has the same combinatorial description in all sufficiently large
degrees. We then use the solution to the Simplified Problem to
construct the solution to the original problem.

\subsection{Outline of the paper} In Section \ref{sec:
  Chow polytopes via triangulations}, we use the identification of the
vertices of the Chow polytope with coherent triangulations to
translate the problem of finding the worst 1-PS 
into an explicit convex optimisation problem. In Section \ref{sec: structure theorem}, we recall the KKT conditions in
convex optimisation and apply them to describe how the stationary
points on each face of the cone $W$ vary with the embedding dimension 
$N$. In Section \ref{sec: corners below the conductor}, we describe
the behavior of the stationary points when all the corners are at or below
the conductor. In Section \ref{sec: corners above the conductor}, we
describe the behavior of the stationary points when there is at least
one corner above the conductor. In Section \ref{sec: conclusion}, we prove the main theorem. Finally, in Section \ref{sec: examples}, we present several examples.

To keep the main body of the paper short, we moved some technical proofs to two appendices. Appendix A gives a detailed proof of Lemmas \ref{lem:
  polynomiality of Q Pk chi psi} and  \ref{lem: formulas for chi and psi}. Appendix B
gives a detailed proof of Propositon \ref{prop: worst 1ps for
    cusps} describing the worst 1-PS for cusps.

\subsection*{Acknowledgements} Our collaboration grew out of a
conversation between the authors,  Jarod Alper, Daniel
Halpern-Leistner, and Frances Kirwan following the February 2021 AIM
workshop ``Moduli problems beyond geometric invariant theory''. The
authors met again at the February 2023 AIM workshop ``Developments in
moduli problems.'' We are grateful to AIM and the workshop
organizers. We also thank Trevor Jones for discussing his preprint
\cite{Jones} with us.

\subsection*{Software} We used several mathematical software
packages for experimentation related to this project, including \GAP, \gfan, \Macaulay,
\MATLAB, \Octave, \QEPCAD, and \SageMath ~\cites{GAP, gfan, Macaulay, MATLAB,
  Octave, QEPCAD, SageMath}. This experimentation was essential to the project
in that it led us to conjecture the main result, though ultimately,
the proof of the main result is a ``pencil-and-paper'' proof that does
not rely on any computer calculations. The examples listed in Table
\ref{tab: Matlab/Octave results} were computed using software. We have
posted our code and some demonstrations for the interested reader on
our website: see \cite{Code}.

\section{The worst 1-PS for Chow points of toric curves as a convex optimisation problem} \label{sec: Chow polytopes via triangulations}

In this section, we translate the problem of finding the worst 1-PS
for Chow points of certain toric curves into an explicit convex
optimisation problem. See Theorem \ref{thm: optimisation problem}.

First, we recall Kempf's description of the worst 1-PS. We follow the
conventions of \cite{HarrisMorrison}*{Ch.~4A} and \cite{Kempf}. 

\subsection{The worst 1-PS} \label{subsec: worst 1-PS}
Let $G$ be a connected reductive algebraic group over a field
$\fieldk$, and let $W$ be a finite-dimensional $k$-vector space. Suppose that $G$ acts on $X \subset \mathbb{P}(W)$ by a
representation $G \rightarrow \GL(W)$. 

Let $\lambda: \mathbb{G}_m \rightarrow G$ be a 1-parameter
subgroup. Write $W_i$ for the
weight $i$ subspace, and write $S_{\lambda}(W)$ for the set of weights
$\{i : W_i \neq 0\}$. Then
\[
W = \bigoplus_{i \in S_{\lambda}(W)} W_i
\]

\begin{definition}
  The Hilbert-Mumford function $\mu(x,\lambda)$ is given by
  \[
\mu(x,\lambda) = \min \{ i \mid x_i \neq 0\}
  \]
\end{definition}

\begin{definition}
We call $\lambda^{*}$ a \emph{worst 1-PS} for $x$ if 
\[
\frac{\mu(x,\lambda^{*})}{\|\lambda^{*}\|}  = \sup_{\lambda} \frac{\mu(x,\lambda)}{\|\lambda\|}
\]
\end{definition}

\textit{A priori} it is not clear that this supremum is finite, or
that it is achieved, but Mumford showed that both statements are
true. Kempf proved that when $x$ is unstable, the supremum is achieved
on a unique parabolic conjugacy class of 1-PS.

\begin{theorem}[\cite{MFK}*{Prop.~2.17}, \cite{Kempf}*{Theorem 3.4}]
Suppose that $x \in X$ is unstable for the action of $G$. Then there
exists a worst 1-PS for $x$.
\end{theorem}

In the following subsections, we recall the definitions of Chow points,
with the goal of understanding their worst one-parameter subgroups.

\subsection{Chow forms and Chow polytopes} \label{subsec: Chow forms}

Given a curve $C \in \PP^N$ of degree $d$, there is a corresponding
point $[C]$ in the Chow variety $\text{Chow}_d(\PP^N)$. We briefly
recall the relevant definitions; see \cite{GKZ} Chapters
3 and 4 for more details. 

Let $X \subset \mathbb{P}^{n-1}$ be an irreducible subscheme of dimension $k-1$ and degree
$d$.  A generic $(n-k-1)$-dimensional projective subspace $L \subset \mathbb{P}^{n-1}$ will miss $X$. Let
\[ \mathcal{Z}(X) = \{L | \dim(L)=n-k-1, X \cap L \neq \emptyset \}
\]

  \begin{theorem}{\cite{GKZ}*{Ch.~3,Prop.~2.1 and 2.2}}
    \begin{enumerate}
\item $\mathcal{Z}(X)$ is an irreducible hypersurface of degree $d$ in $\Gr(n-k,n)$.      
\item Up to a scalar, $\mathcal{Z}(X)$ is defined by a polynomial $R_X$ called the
  \emph{Chow form} of $X$.
   \end{enumerate}
\end{theorem}

Thus, $X \subset \mathbb{P}^{n-1}$ corresponds to the point $[R_X]
\in \mathbb{P}(\Sym^d \Lambda^{n-k} \, \fieldk^{n})$. We will write $[X]$ for
$[R_X]$.

The closure of the set of points of the form $[X]$ has the structure of an algebraic
variety called the \emph{Chow variety}.

\subsection{Numerical semigroups and monomial curves} \label{subsec: Semigroups}

We seek the worst 1-PS for the Chow points of curves $[C] \in
\text{Chow}_d(\mathbb{P}^N)$. In general, this is difficult, so in
this work we only consider the special case where $C$ is a toric rational curve with a single unibranch singularity at $p \in C$. One can associate to such a curve its semigroup of values.
\begin{definition} The semigroup of values of $C$ is the numerical semigroup
\[\Gamma_C:= \{n\in \NN \mid \exists f\in \mathscr{O}_{C,p} \text{ with  } \nu_p(f) = n \} \subset \NN \] where  $\nu_p$ is the valuation at the singular point $p\in C$.
\end{definition}

\begin{example}
The simplest example is a cuspidal curve, that analytically looks like
$y^2 = x^3$ and has semigroup of values $\Gamma  = \langle 2,3\rangle
= \{0,2,3,4,\ldots\}$. More generally we may consider higher order
cusps $y^2 = x^{2r+1}$, and get $\Gamma  = \langle 2,2r+1\rangle =
\{0,2,4,\ldots,2r,2r+1,2r+2,\ldots\}$
\end{example} 

For each numerical semigroup $\Gamma$, there exists a positive integer
$\cond(\Gamma)$ called the \emph{conductor} of $\Gamma$ such that
$\NN_{\geq \cond(\Gamma)} \subseteq \Gamma$. Thus one has
$\mid\NN\setminus \Gamma\mid < \infty$, and the size of that set,
i.e.~the number of gaps, is equal to the arithmetic genus of the curve
$C$ (or equivalently, the $\delta$-invariant of the singularity). We list the elements of $\Gamma $ as
$\gamma_0,\gamma_1,\gamma_2,\ldots$ with $\gamma_0 = 0$, and write
$\condIndex(\Gamma)$ for the index of the conductor. That
is, $\gamma_{\condIndex(\Gamma)} = \operatorname{cond}(\Gamma)$.

Conversely, given a numerical semigroup $\Gamma\subset \NN$,  we can
write down a curve $C_\Gamma \subset \PP^N$ of degree $d$ that has
semigroup of values $\Gamma$.

\begin{definition} \label{def: monomial curve}
  Let $\Gamma\subset \NN$, and let $d>\cond(\Gamma)$.  The \emph{monomial curve} $C_\Gamma$ is the closure of the image of the parametrisation
\[t \longmapsto (1,t^{\gamma_1},t^{\gamma_2},\dots,t^d)\]
where $\gamma_i$ are the elements of $\Gamma_{\leq d}$. 
\end{definition}

\subsection{The $\mathbb{G}_m$-action on $C_\Gamma$}

Let $\Gamma$ be a numerical semigroup, and let $N > \condIndex(\Gamma)$. Let $C_\Gamma \subset
\PP^N$ be the monomial curve as in Definition \ref{def: monomial curve}. Our goal is to find
the worst 1-PS for this action of $\GL_{N+1}$ on the Chow point
$[C_\Gamma]$.

Hyeon and Park have shown that for an unstable
point $x$ in a GIT problem, a generic maximal torus will not contain a
destabilizing 1-PS \cite{HP}. Fortunately, we have a  clue where
to look for the worst 1-PS:  $C_\Gamma$ has a
$\mathbb{G}_m$-automorphism scaling the coordinate $t$, and it acts
with distinct weights on $\PP^N$. This allows us to apply the
Kempf-Morrison Lemma \cite{MS}*{Prop.~4.7} and conclude that if $C_\Gamma$ is
unstable for the $\GL_{N+1}$, then a worst 1-PS will appear in the
maximal torus diagonalizing the $\mathbb{G}_m$-action. (Note: the statement in
the cited work is for a finite automorphism group, but the proof works
for $\mathbb{G}_m$ as well.)


$C_\Gamma$ is a non-normal toric variety. Its associated polytope is just the interval $[0,d]$. Because
the variety is not normal, some of the interior lattice points of the
polytope are missing---these are exactly the gaps in the semigroup.

This observation is useful because the Chow polytope of a toric variety
has a second description due to Kapranov, Sturmfels, and Zelevinsky: it is the \emph{secondary polytope} of the
polytope of $X$ (\cite{GKZ}*{Ch.~8, Thm.~3.1}). In the next section,
we briefly review part of this theory.

\subsection{The Hilbert-Mumford function for Chow points of toric varieties}
Consider a polytope $Q = \text{conv}(A)$, where $A = \{A_0,\dots , A_n\} \subset \RR^{k-1}$ is a finite set of vectors.

\begin{definition} The \emph{secondary polytope} $\Sigma(A)$ of $Q$ is the
  convex hull of the vectors $\phi_T$ as $T$ runs over all
  triangulations of $Q$.
\end{definition}

Let $\RR^{A}$ be the set of functions $A \rightarrow \mathbb{R}$. Given a triangulation $T$ of $Q$, one can associate a vector  $\phi_T \in \RR^{A}$, whose $i$th entry is the real number \[\phi_T(i) = \sum\limits_{\sigma:a_i \in \text{Vert}(\sigma)} \text{Vol}(\sigma)  \] where $\text{Vol}$ is a translation-invariant volume form and the sum is over all maximal simplices of $T$ for which $A_i$ is a vertex. If $A_i$ is not a vertex of any maximal simplex of $T$, the entry is zero. 

The vertices of the
secondary polytope are those vectors $\phi_T$ corresponding to
coherent triangulations of $(Q,A)$ \cite{GKZ}{Ch.~7, Thm.~1.7} (We omit the
definition of coherence, because we won't need it: in dimension one,
all triangulations are coherent.)

The following definitions and lemma are adapted
from \cite{GKZ}*{Ch.~7 Lem.~1.9(c)}. (The cited result gives a formula
for the maximum, but the minimum appears in our GIT calculations.)

\begin{definition}
Given a function $f \in \RR^A$, let $H_f$ be the convex
hull of the vertical half-lines \[\{(x,y) \mid y \geq f(x), x\in A, y
  \in \RR\}.\]

Let the \emph{lower convex hull of $f$}, denoted $\lch(f): Q \rightarrow R$, be the piecewise linear function
\[\lch(f)(x) = \min \{y \mid (x,y) \in H_f \}.\]  
\end{definition}

\begin{lemma} \cite{GKZ}*{Ch.~7 Lem.~1.9(c)} \label{lem min pairing is
    integral} Given a polytope $Q = \text{conv}(A)$ and a vector $w
  \in \RR^{n+1}$ we have \[\min_{\phi \in \Sigma(A)} \langle w,\phi
    \rangle  = k \int_Q \lch(f_w) (x) dx,\]
  where $f_w: A \rightarrow \RR$ is the function $f_w(A_i) = w_i$. 
\end{lemma} 

In summary: to compute the miminal pairing of a vertex of the secondary polytope with a vector $w \in \RR^{n+1}$, we first take the lower convex hull, and then integrate it over the polytope.

The formula in Lemma \ref{lem min pairing is integral} leads to the
following expression for Hilbert-Mumford function. To our knowledge,
this has not appeared in the literature before.

\begin{proposition}\label{prop HM is lch}
The value of the Hilbert-Mumford function for the point $[X] \in
\Chow_d(\PP^N)$ at a 1-PS $ \lambda : \mathbb{G}_m \rightarrow T$ with
weight vector $w$ is \[\mu(\lambda,[X]) = k \int_Q\lch(f_w)\]

\end{proposition}

\subsection{Chow polytopes of monomial curves $C_\Gamma$}

Now we apply these results to monomial curves.

Let $\Gamma$ be a numerical semigroup, and let $N \geq \condIndex
\Gamma$, and let $C_\Gamma \subset \PP^N$ be the corresponding a monomial curve

Its corresponding polytope is $Q =\operatorname{conv}(A) = [0,d]$, where $A =
(\gamma_0,\ldots,\gamma_N)$.
The Chow polytope of $C_\Gamma$ is the same as the secondary polytope $\Sigma(A)$, and hence its vertices correspond to triangulations of $(Q,A)$. A 1-dimensional triangulation is just a decomposition into intervals, and it determined by the placement of the vertices; hence, such triangulations are in bijection with subsets of the finite set $A \setminus \{0,\gamma_N\}$.

\begin{example} 
Consider a cuspidal cubic $X^2Z-Y^3 \subset \PP^2$. The secondary polytope/Chow polytope has one vertex for each integral subdivision of $[0,3]$, where we are not allowed to place a vertex at $1$, because that value is missing from the semigroup of values. Thus there are only two possible sudivisions: $[0,2,3]$ and $[0,3]$, i.e.~we either place a vertex at $2$ or not. These two subdivisions correspond to the two vertices of the weight polytope, and hence to the two possible initial ideals: one is $X^2Z$ and the other is $Y^3$.

\end{example}

Thus for curves, Proposition \ref{prop HM is lch} says that the value of the Hilbert-Mumford function is the area under the graph of the lower convex hull of the piecewise linear function determined by the weight vector $w$.

\begin{definition} We call a weight vector \emph{convex} if it is equal to its lower
  convex hull.
\end{definition}

Equivalently, a weight vector is convex if and only if the slopes of the corresponding piecewise linear function are increasing.

\begin{lemma} \label{lem: worst 1ps is convex}
Let $w^{*}$ be the weight vector of the worst 1-PS for $C_{\Gamma}$. Then
$w^{*}$ is nonnegative and convex.
\end{lemma}
\begin{proof}
First, suppose that $w$ has at least one negative coordinate.   
The vertices of the secondary polytope all lie in the positive
orthant, and the Chow polytope coincides with the secondary
polytope. Let $w'$ be defined by $w'_i = \max\{0,w_i\}$. Then $\langle
w',\phi \rangle \geq \langle w,\phi \rangle $ for all $\phi \in
\Sigma(A)$. Hence $\mu([C_{\Gamma}],w') \geq \mu([C_{\Gamma}],w)$. 
Also $\| w' \| < \| w \|$. Thus
\[
\frac{\mu([C_{\Gamma}],w')}{\| w' \|} \geq \frac{\mu([C_{\Gamma}],w)}{\| w \|}. 
\]
Thus, $w^{*}$ is nonnegative.

Now suppose $w$ satisfies $w_i \neq 0$ for all $i$, and for at least
one index $i$, we have $w_i > \lch(f_w)(A_i)$. Then define $w'$ by
lowering the $i^{th}$ coordinate to the lower convex hull. We have
$\mu([C_{\Gamma}],w') =  \mu([C_{\Gamma}],w)$ since they have the same
lower convex hull, but $\| w' \| < \| w \|$. Thus 
\[
\frac{\mu([C_{\Gamma}],w')}{\| w' \|} \geq \frac{\mu([C_{\Gamma}],w)}{\| w \|}. 
\]
Thus, $w^{*}$ is convex.
\end{proof}

\subsection{The optimisation problem} \label{subsec: optimisation problem}

In this section we present an explicit optimisation problem whose
solution is a worst 1-PS for $[C_{\Gamma}]$. We need to introduce a
little more notation first.

\subsubsection{The cone $W$.} 
By Lemma \ref{lem: worst 1ps is convex}, in finding the worst 1-PS
we are justified in restricting our attention to one-parameter subgroups whose weight vectors $w$ are convex.

\begin{definition}
  Let $W\subset \RR^{N+1}$ be the set of convex weight vectors. 
\end{definition}

Since a weight vector is convex if and only if the slopes of the
corresponding piecewise linear function are increasing, $W$ is a
  polyhedral cone generated by the inequalaties that the slope of the
  $i$th line less than or equal to the slope of the $(i+1)$th line. For more
  details about $W$, see Section \ref{subsubsec: W} below.

\subsubsection{The vector $a$.} 

\begin{lemma}  \label{lem: area is dotting with a}
Let $C_\Gamma$ be a monomial curve, let $A =
(\gamma_0,\ldots,\gamma_N)$, $Q = \operatorname{conv}(A)$, and
suppose that $w$ is nonnegative and convex. Then
\[
2 \int_Q \lch(f_w) (x) dx = a \cdot w
\]
where
    \[
      a_i = \left\{
\begin{array}{ll}
\gamma_1 & \text{if $i=1$} \\
  \gamma_{i} - \gamma_{i-2} & \text{if $2 \leq i \leq N$} \\
  1 & \text{if $i=N+1$} 
\end{array}
    \right.
  \]
(Note: in this formula, we have indexed $w = (w_0,\ldots,w_N)$ and $A
= (\gamma_0,\ldots,\gamma_N)$ starting at 0,  and indexed $a =
(a_1,\ldots,a_{N+1})$ starting at 1. This is because in the sequel, $a$ will appear in a matrix equation where we index the rows beginning
at 1.)  
\end{lemma}
\begin{proof}
Since $w$ is convex, we have $\lch(f_w) (\gamma_i) = w_i$. Since $w$ is
nonnegative, the integral $2\int_{\gamma_{i-1}}^{\gamma_{i}} \lch(f_w)
(x) dx$ is twice the area of the trapezium with heights $w_{i-1}$ and
$w_{i}$ and width $(\gamma_{i} - \gamma_{i-1})$. Then 
\begin{align*}
2 \int_{0}^{\gamma_N} \lch(f_w) dx &= \sum_{i=1}^{N-1} (w_{i-1}+w_{i})
                                     (\gamma_{i} - \gamma_{i-1})\\
  &= w_0 (\gamma_1-\gamma_0) + \sum_{i=2}^{N} w_{i-1}
    (\gamma_{i}-\gamma_{i-2}) + w_N (\gamma_N - \gamma_{N-1})\\
  &=w_0 (\gamma_1) + \sum_{i=2}^{N} w_{i-1}
    (\gamma_{i}-\gamma_{i-2}) + w_N (1).
\end{align*}

Here the last line follows because $\gamma_0 = 0$ and $(\gamma_N -
\gamma_{N-1})=1$, since $N \geq \condIndex(\Gamma)$.

\end{proof}

\begin{remark}
	Once we are above the index of the conductor, all gaps between consecutive semigroup elements are $1$, and hence, if the degree of the curve is large enough, the tail of $a$ looks like $(2,2,\dots,2,1)$. We can think of increasing the degree as merely adding an extra $2$. This gives the moral reason why we might expect to see persistent optima, and motivates the Simplified Problem, discussed in the next subsection.
\end{remark}

We are now ready for the main result of this section.

\begin{theorem} \label{thm: optimisation problem}
Let $W$ and $a$ be the cone and vector defined above for the monomial
curve $C_\Gamma$. 
  
Let $\prox(a,W)$ be the nearest point on
the cone $W$ to the vector $a$ outside it.
  
Then $\prox(a,W)$ is the weight vector of a worst 1-PS for $[C_{\Gamma}]$. 
\end{theorem}

\begin{proof}
By definition, the worst 1-PS satisfies
\[
\frac{\mu(x,\lambda^{*})}{\|\lambda^{*}\|}  = \sup_{\lambda} \frac{\mu(x,\lambda)}{\|\lambda\|}
\]
By applying Lemma \ref{lem: worst 1ps is convex}, then Proposition
\ref{prop HM is lch}, and then Lemma \ref{lem: area is dotting with
  a}, we have
   \begin{align*}
\frac{\mu([C_\Gamma],w^{*})}{\|w^{*}\|}  &= \max_{w \in W}
     \frac{\mu([C_\Gamma],w)}{\|w\|} \\
& = \max_{w \in W} \frac{2 \int_Q \lch(f_w) (x) dx}{\|w\|} \\
& = \max_{w \in W} \frac{a \cdot w}{\|w\|} \\
&= \max_{w \in W} \|a \| \cos(\theta(a,w))     
\end{align*}

Since $\prox(a,W)$ minimizes $\theta(a,w)$, the result follows.

\end{proof}



\subsubsection{The Simplified Problem}
In the sequel, it will be convenient to solve the optimisation problem
for the vector $a$ whose last coordinate is 2 rather than 1. We call this the Simplified
Problem. We then use the solution to the Simplified Problem to
construct a solution to the original problem, which we henceforth
refer to as the Unsimplified Problem.

We illustrate an example in Figure \ref{fig:graphic for simple cusp}. Let $\Gamma = \langle 2, 3 \rangle$, and
consider the Simplified Problem when $N = 10$. Then we have $a = (2,
3, 2, 2, 2, 2, 2, 2, 2, 2, 2)$, and the worst 1-PS has weights \\$w =
\left(\frac{33}{14}, \frac{157}{70}, \frac{153}{70}, \frac{149}{70}, \frac{29}{14}, \frac{141}{70}, 2, 2, 2, 2, 2\right)$. We
plot the points $(\gamma_i,a_i)$ in green, and plot a blue line graph
through the points 
$(\gamma_i, w_i)$.

\begin{figure}[H]
\centering
\begin{tikzpicture}[scale=1.0] 
  \draw[->] (0,0) -- (12,0);
  \draw[->] (0,0) -- (0,4);
  \draw (-0.1,2)--(0.1,2);
  \draw (-0.1,2) node[anchor=east] {$2$};
  \foreach \x in {2,4,6,8,10}
    \draw (\x,-0.1)--(\x,0.1);
  \foreach \x in {2,4,6,8,10,12}
    \draw (\x,-0.1) node[anchor=north] {$\x$};  
  \filldraw[green] (0,2) circle (0.05);
  \filldraw[green] (2,3) circle (0.05);
  \foreach \x in {3,4,5,6,7,8,9,10,11} 
    \filldraw[green] (\x,2) circle (0.05);
  \draw[blue] (0, 2.357)--(6,2.014)--(7,2)--(11,2);
\end{tikzpicture}
  \caption{The vector $a$ and the optimal weights $w$ for the
    Simplified Problem for $ \Gamma = \langle 2, 3 \rangle$ and
    $N=10$}
    \label{fig:graphic for simple cusp}
\end{figure}
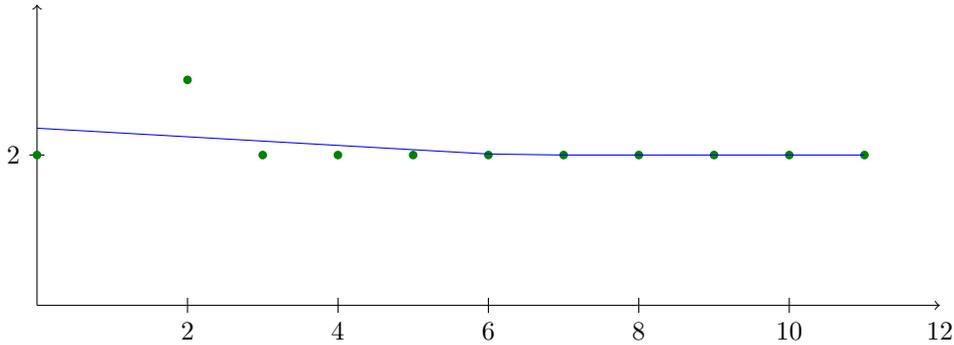

\section{The KKT matrix equation and its solutions} \label{sec:
  structure theorem}

To solve the optimisation problem described in the previous section, we will use the Karush-Kuhn-Tucker (KKT) conditions in nonlinear optimisation to study the closest point on the span of each face of $W$ to the vector $a$.

Recall that we write $\operatorname{cond}(\Gamma)$ for the
conductor of $\Gamma$, and $\condIndex(\Gamma)$ for its index. That
is, $\gamma_{\condIndex(\Gamma)} = \operatorname{cond}(\Gamma)$. When we say there is a corner above the conductor, we mean that the set of corner indices $I$ contains an element $\ell$ greater than $\condIndex(\Gamma)$, so that the point $(\gamma_\ell,w_\ell)$ is  to the right of the line $x = \operatorname{cond}(\Gamma)$.

\subsection{The KKT conditions for convex optimisation} \label{subsec: KKT}

The KKT conditions are necessary conditions satisfied
by an optimal point for a broad class of
nonlinear optimisation problems. See for instance
\cite{BoydVandenberghe}*{Section 5.5.3}.  In many cases, including the problem
considered here, they are also sufficient.

Suppose we want to minimize a function $f(w)$ subject to the
constraints $g_i(w) \leq 0$ and $h_j(w) = 0$. The KKT conditions for
the optimal point $w^*$ are as follows. 

\begin{enumerate}
\item Stationarity: $\nabla f(w^{*}) + \sum \lambda_j \nabla h_j(w^*) + \sum \mu_i \nabla g_i(w^*) = 0$
\item Primal feasibility: $g_i(w^*) \leq 0$ and $h_j(w^{*})=0$ for all $i,j$
\item Dual feasibility: $\mu_i \geq 0$ for all $i$
\item Complementary slackness: $\mu_i g_i(w^{*}) = 0$ for all $i$
\end{enumerate}

In our case (computing the nearest
point on a polyhedron to a point outside it), the
objective function is a strongly convex quadratic function, and the
constraints are given by affine inequalities. Therefore, the KKT
conditions are also sufficient for obtaining an optimal $w^*$ (see for
instance 
\cite{BoydVandenberghe}*{Section 5.5.3}). Finally, the optimal $w^*$ is unique,
since the objective function is strongly convex.

\subsubsection{The polyhedral cone $W$ associated to
  $\gamma$} \label{subsubsec: W}

For any $\gamma$, the cone $W$ is defined by the conditions that the slopes of the line
segments connecting the points $(\gamma_i,w_i)$ are increasing. This
yields $N-1$ halfspaces in $\RR^{N+1}$. The polyhedral cone $W$ therefore has
a two-dimensional lineality space. The faces of $W$ are easy to
describe. A vector $w$ lies on the facet $m_{i} = m_{i+1}$ if and only
there is no corner in the graph  at $(\gamma_{i},w_{i})$.

The inequality
\[
\frac{w_{i}-w_{i-1}}{\gamma_{i}-\gamma_{i-1}} \leq \frac{w_{i+1}-w_{i}}{\gamma_{i+1}-\gamma_{i}}
\]
rearranges to
\[
-(\gamma_{i+1}-\gamma_{i})w_{i-1}
+(\gamma_{i+1}-\gamma_{i-1})w_{i}-(\gamma_{i}-\gamma_{i-1})w_{i+1}
\leq 0.
\]

The following piecewise linear functions are useful when working with $W$.

\begin{definition}
  \begin{align*}
    F_{k}(x) & := \max\{-x+k,0\}\\
    L_{\gamma_N}(x) &:= \gamma_N-x \\
    L_{1}(x) &:=1
  \end{align*}
 For a vector $\gamma$, we abuse notation and write $F_{k}(\gamma)$
 to denote the vector with coordinates $F_{k}(\gamma_i)$ for all
 $i$. (We do this also for $L_{\gamma_N}(x)$ and $L_{1}(x)$.) 
\end{definition}

\begin{proposition} \label{prop:rays of W}
  \begin{enumerate}
\item Any piecewise linear function $F(x)$ on $[0,\gamma_N]$ with corners lying over integers $0 \leq
  n \leq d$ is a linear combination of the functions $\{F_{k}(x): 1
  \leq k\leq \gamma_n-1\}$, $L_{\gamma_N}(x)$, and $L_{1}(x)$.
\item The lineality space of $W$ is spanned by $L_{\gamma_N}(\gamma)$ and
  $L_{1}(\gamma)$.
\item For each $1 \leq i \leq N-1$, the vector $F_{\gamma_i}(\gamma)$ spans a ray of $W$.
\item Let $F(x)$ be a piecewise linear function on $[0,\gamma_N]$ with
  corners lying over integers $n \in \gamma$. Suppose that the slopes of the line segments in the graph
  of $F$ are
  negative and increasing, and $F(x) \geq 0$. Then $F(x)$ is a
  nonnegative linear combination of the rays spanned by 
  $\{F_{\gamma_i}(\gamma) : 1 \leq i \leq N-1\}$ and $L_{\gamma_N}(\gamma)$
  and $L_{1}(\gamma)$.
\end{enumerate}
\end{proposition}

\subsection{The KKT matrix equation}

\begin{definition} \label{def: KKT matrix equation}
We define the \emph{KKT matrix equation for $\face_{\gamma(I)}$} as follows.

First, we define an $(N+1) \times (N+1)$ matrix $A$ as follows. For $1
\leq j \leq N-1$, if $j
\in I$, then column $j$ of $A$ is given by $2F_{\gamma_j}(\gamma)$. If
$j \not\in I$, then column $j$ of $A$ is given by the coefficient
vector of the inequalities $g_j(w) \leq 0$. The last two columns are given by $2L_{\gamma_N}(\gamma)$ and
$2L_{1}(\gamma)$. 

Next, we define a vector $x \in \mathbb{R}^{N+1}$ as follows. For $1
\leq i \leq N-1$, if $i
\in I$, then $x_i$ is one of the parameters $t_i$ used to parametrize
$\Span \face_{\gamma(I)}$. If $i \not\in I$, then $x_i$ is one of the KKT
multipliers $\mu_i$. The last two coordinates $x_{N}$ and $x_{N+1}$
are the parameters used for the lineality space. 

Then the KKT matrix equation for $\face_{\gamma(I)}$ is 
\[ Ax = 2a.
\]
\end{definition}

We have two versions of the vector $a$: one for the Simplified
Problem, and one for the Unsimplified Problem. Most of our discussion
will be dedicated to the Simplified Problem. Then, at the end, we
will show how the result for the Simplified Problem implies the result
for the Unsimplified Problem.

The matrices $A$, $x$, and $a$ all depend on $N$, $\Gamma$, and $I$. However, since $\Gamma$ and $I$ are typically clear from context, we do not show them in our notation.

We record explicit formulas for the matrix $A$ and the vector $a$.

\begin{proposition} \label{prop: KKT matrix entries}
    The entries of $A$ are given by the following formula.
    \[
      A_{i,j} = \left\{
\begin{array}{ll}
\gamma_j - \gamma_{j+1} & \text{if $j\leq N-1$, $j \not\in I$, and $i=j$} \\
  \gamma_{j+1} - \gamma_{j-1} & \text{if $j\leq N-1$, $j \not\in I$, and $i=j+1$} \\
  \gamma_{j-1}- \gamma_{j} & \text{if $j\leq N-1$, $j \not\in I$, and $i=j+2$} \\
  2(\gamma_j - \gamma_{i-1}) & \text{if $j\leq N-1$, $j \in I$, and $i \leq j$} \\
  2(\gamma_N - \gamma_{i-1}) & \text{if $j= N$} \\
  2 & \text{if $j= N+1$} \\
  0 & \text{otherwise}
\end{array}
    \right.
    \]
    The vector $a$ on the right hand side of the KKT matrix equation is given by the following formula.
    \[
      a_i = \left\{
\begin{array}{ll}
\gamma_1 & \text{if $i=1$} \\
  \gamma_{i} - \gamma_{i-2} & \text{if $2 \leq i \leq N$} \\
  2 & \text{if $i=N+1$ (Simplified Problem)} \\
  1 & \text{if $i=N+1$ (Unsimplified Problem)} 
\end{array}
    \right.
    \]
\end{proposition}

The application to our problem is as follows.

\begin{proposition} \label{prop: verify proximum}
Let $x$ be the solution
of the KKT matrix equation for $\face_{\gamma(I)}$.  If $x_i \geq 0$
for all $1 \leq i \leq N-1$, then
$w(t_1,\ldots,t_{k+2})$ is the closest point on $W$ to $a$.

Moreover, if $x_i > 0$ for all $i \in I$, then $\face_{\gamma(I)}$ is
the smallest face of $W$ containing this point.
\end{proposition}

\subsection{Persistence}

We begin with the following simple observation.
\begin{lemma}
Fix $\Gamma$, $I$, and $N$. Suppose that $x$ is a solution to the KKT matrix equation for the
Simplified Problem with $x_{N} = 0$ and $x_{N+1} = 2$.
Define $x'$ by 
\[
  x_i' = \left\{
\begin{array}{ll}
  x_i, & i \leq N \\
  0, & i = N+1 \\
  2, & i = N+2
\end{array}
\right.
\]
Then $x'$ is a solution to the KKT matrix equation for $\Gamma$ and
$I$ with embedding dimension $N+1$.
\end{lemma}

This motivates the following definition.

\begin{definition}
We call $x$ a \emph{persistent solution to the Simplified Problem} for
the face associated to $I$ if
$x_N = 0$ and $x_{N+1} = 2$. 
\end{definition}

We aim to prove the following result.

\begin{theorem*}[Persistence of the global optimum for the Simplified Problem]
  Let $\Gamma$ be a numerical semigroup. There exists an integer $N_0^{\simp}$ and a
set of corner indices $I^{\simp}$ (both depending on $\Gamma$) such that for all $N \geq N_0^{\simp}$, the global
optimum for the Simplified Problem for $\Gamma$ is the persistent
solution for the face of $W$ corresponding to $I^{\simp}$.
\end{theorem*}

For small values of $N$, there can be non-persistent optima. However, our strategy is to show that such non-persistent optima obey bounds on $N$. Thus, for large $N$, the only remaining possibility is that there is a persistent optimum.

\subsection{How the stationary points vary with $N$} 
Let $I$ be a set of corners. For each $N$, let $x(N)$ be the solution to the KKT
matrix equation for the corresponding face of $W$.

By Cramer's Rule, we have a formula for each coordinate in the
solution $x(N)$ in terms of determinants.
\[
x_j(N) = \frac{\det A(N,j)}{\det A(N) }.
\]

The last three coordinates $x_{N-1}$, $x_N$, and $x_{N+1}$ play a
special role in the discussion, and so we have special notation for
their numerators. 

Recall that under our notation conventions, $x_N$ and $x_{N+1}$ are
parameters. $x_{N+1} = w_N$ is the $y$-value of the last point on the
graph of $w$, and $-x_N$ is the slope of the last line segment in the
graph of $w$. Neither of these two quantities is required to be
nonnegative---they parametrize the lineality space of the polyhedron $W$.

\begin{definition}
  Define functions
  \begin{align*}
    Q(N) &= (-1)^{N+1} \det A(N) \\
    P_j(N) &= (-1)^{N+1} \det A(N,j) \\
    \chi(N) &= (-1)^{N+1} \det A(N,N-1) \\
    \psi(N) &= (-1)^{N+1} \det A(N,N) \\
    \omega(N) &= (-1)^{N+1} \det A(N,N+1) \\
  \end{align*}
\end{definition}

  (This notation is chosen as a mnemonic device. Recall that in the Greek alphabet, the last three letters in order
  are $\chi$, $\psi$, and $\omega$.)

  Then we have
  \begin{displaymath}
    \begin{array}{ccccccc}
   x_j &=& \displaystyle \frac{\det A(N,j)}{\det A } &=& \displaystyle \frac{(-1)^{N+1}\det A(N,j)}{(-1)^{N+1}\det A } &=& \displaystyle \frac{P_j}{Q}. \\
    x_{N-1} &=& \displaystyle \frac{\det A(N,N-1)}{\det A } &=& \displaystyle \frac{(-1)^{N+1}\det A(N,N-1)}{(-1)^{N+1}\det A } &= &\displaystyle \frac{\chi}{Q}. \\
    x_{N} &=& \displaystyle \frac{\det A(N,N)}{\det A } &=& \displaystyle \frac{(-1)^{N+1}\det A(N,N)}{(-1)^{N+1}\det A } &= & \displaystyle \frac{\psi}{Q}. \\
    x_{N+1} &=& \displaystyle \frac{\det A(N,N+1)}{\det A } &=& \displaystyle \frac{(-1)^{N+1}\det A(N,N+1)}{(-1)^{N+1}\det A } &=& \displaystyle \frac{\omega}{Q}. 
\end{array}
\end{displaymath}
  
Finally, for $j \leq N$, we define $A'(N,j)$ as the matrix obtained by replacing column $j$ in $A(N)$ by $2a-\underline{2}$. Since the last column of $A(N,j)$ is $\underline{2}$, the matrices $A(N,j)$ and $A'(N,j)$ are column equivalent, and thus $\det A(N,j) = \det A'(N,j)$.

\begin{lemma} \label{lem: polynomiality of Q Pk chi psi}
  \begin{enumerate}
 Suppose    $N \geq \max \{ \condIndex(\Gamma)+2,\max(I) +2 \}$.
\item  The functions $Q(N)$, $\chi(N)$, $\psi(N)$, and $\omega(N)$ are polynomials with the following degree bounds. 
  \begin{enumerate}
  \item $\deg Q = 4$
  \item $\deg \chi \leq 3$
  \item $\deg \psi \leq 2$
  \item $\deg \omega = 4$.
  \end{enumerate}
\item  If  $N \geq \max \{ \condIndex(\Gamma)+2,\max(I) +2 \}$ and $N \geq j+1$, the function $P_j(N)$ is a polynomial of degree at most 4.
  \end{enumerate}
\end{lemma}

We give a full proof in Appendix A. Here we only sketch the proof. As $N$ grows, the matrices $A(N)$ and $A'(N,j)$ grow in a way that is easy to describe. We can write recurrences for the matrices, then show that
the appropriate difference equations vanish to establish that each of
these functions is a polynomial in $N$ with the degree bounds claimed.  (This is a paper-and-pencil proof checked by computer---the proof does not rely on computer calculations.)

\begin{corollary}
Let $x$ be the solution of the KKT matrix equation for face $I$. Then each coordinate of $x$ is a rational function in $N$.
\end{corollary}

  Fix a positive integer $k$. We use the following notation for the Taylor expansions of these polynomials centered at $k$.
  \begin{eqnarray*}
    \chi(N) &=& \chi_3 (N-k)^3 + \chi_2 (N-k)^2 +  \chi_1 (N-k) + \chi_0 \\
    \psi(N) &=& \psi_2 (N-k)^2 +\psi_1 (N-k)+ \psi_0
  \end{eqnarray*}

When $k$ is sufficiently large, we have explicit formulas for these coefficients in terms of minors of the matrices $A'$.
\begin{lemma} \label{lem: formulas for chi and psi}
  Suppose that $k \geq \max \{ \condIndex(\Gamma)+2,\max(I) +2 \}$.
\begin{enumerate}
\item
  \begin{align*}
    \chi_3 & = \frac{1}{3}(-1)^{k+1} \left(2 \del_{k,k+1}^{k,k+1} A'(k,k-1) + 2 \del_{k-1,k+1}^{k,k+1} A'(k,k-1) \right) \\
    \chi_2 & = (-1)^{k+1} \left(4 \del_{k,k+1}^{k,k+1} A'(k,k-1) + 2 \del_{k-1,k+1}^{k,k+1} A'(k,k-1) \right) \\
    \chi_1 &= \frac{1}{3}(-1)^{k+1} \left( -14 \del_{k,k+1}^{k,k+1} A'(k,k-1) -8 \del_{k-1,k+1}^{k,k+1} A'(k,k-1) - 6 \del_{k+2}^{k+1} A'(k+1,k) \right)\\
    \chi_0 &= (-1)^{k+1} \det A'(k,k-1).
  \end{align*}
\item 
  \begin{align*}
    \psi_2 & = (-1)^{k+1} \left(\del_{k+1}^{k+1}  A'(k,k) + \del_{k}^{k+1}  A'(k,k) \right) \\
    \psi_1 & = (-1)^{k+1} \left(3\del_{k+1}^{k+1}  A'(k,k) + \del_{k}^{k+1}  A'(k,k) \right) \\
    \psi_0 &= (-1)^{k+1} \det A'(k,k).
  \end{align*}
\end{enumerate}
\end{lemma}

  We give a full proof in Appendix A. Here we only sketch the
  proof. We can get the leading coefficients $\chi_3$ and $\psi_2$
  using difference equations via a procedure similar to the proof of  \ref{lem: polynomiality of Q Pk chi psi}. We can solve for the remaining coefficients by interpolating these polynomials using their values when $N=k$, $N=k+1$, $N=k+2$.

Next, we give a formula for $x_j$.
\begin{lemma} \label{lem: Pj}

  For each $j$ satisfying $j \geq \max \{ \condIndex(\Gamma)+2,I+2 \} $ and $j< N$, we have
  \[
  x_j = (N-j)(N-j+1) \left(  \frac{1}{2}x_{N-1} + \frac{1}{3}(N-j-1)x_N \right)
  \]
  Hence, after clearing the common denominator $Q$, we have
    \begin{equation} \label{eqn: Pj}
  P_j = (N-j)(N-j+1) \left(  \frac{1}{2} \chi + \frac{1}{3}(N-j-1)\psi \right)
  \end{equation}
\end{lemma}

\begin{proof} For sufficiently large $N$, the lower right portion of the KKT matrix is
\[ \begin{smallmatrix}
\ddots&  \\ 
&-1&0&0&0&0&0&0&0&18&2\\
&2&-1&0&0&0&0&0&0&16&2\\
&-1&2&-1&0&0&0&0&0&14&2\\
&0&-1&2&-1&0&0&0&0&12&2\\
&0&0&-1&2&-1&0&0&0&10&2\\
&0&0&0&-1&2&-1&0&0&8&2\\
&0&0&0&0&-1&2&-1&0&6&2\\
&0&0&0&0&0&-1&2&-1&4&2\\
&0&0&0&0&0&0&-1&2&2&2\\
&0&0&0&0&0&0&0&-1&0&2
\end{smallmatrix} 
\]

Recursively row reducing from the last row upward yields
\[
-x_{N-j} + \frac{1}{3}(j^3-j)x_N + (j^2+j)x_{N+1} = 2(j^2+j)
\]

We rearrange this equation, reindex, and substitute $x_{N-1} = 2x_{N+1}-4$ to prove the claim.
\end{proof}

This leads to an identity between the coefficients of the polynomials $\chi$ and $\psi$. 
\begin{lemma} \label{lem: chi3 psi2 relation}
$\frac{1}{2}\chi_{3} = -\frac{1}{3} \psi_2$.
  \end{lemma}
  \begin{proof}
For each $j$ with $j \geq \max \{ \condIndex(\Gamma)+2,I+2 \} $ and $j< N$, the polynomial $P_j$ has degree at most 4. Thus the degree 5 coefficient in the expression (\ref{eqn: Pj}) must vanish.
\end{proof}

We will also use the following lemma.

\begin{lemma} \label{lem: crossing implies not globally optimal}
Suppose that the graph of $w$ crosses the line $y=2$ after the
conductor of $\Gamma$. Then $w$ is not globally optimal for the
Simplified Problem.
  
\end{lemma}
\begin{proof}  
  If $w_0 < 2$, then the graph of $w$ can only cross the line $y=2$
  once.  Such a $w$ is not globally optimal because the line $y=2$
  performs better. In pictures:
\begin{figure}[H]
\begin{center}
\begin{tikzpicture}[scale=0.35] 
  \draw[->] (0,0) -- (15,0);
  \draw[->] (0,0) -- (0,5);
  \draw[thin] (0,2)--(15,2);
  \draw (-0.1,2)--(0.1,2);
  \draw (-0.1,2) node[anchor=east] {$2$};
  \draw (6,-0.1)--(6,0.1);
  \draw (6,-0.1) node[anchor=north] {$c$};  
  \filldraw[green] (0,2) circle (0.05);
  \filldraw[green] (2,4) circle (0.05);
  \filldraw[green] (4,4) circle (0.05);
  \filldraw[green] (6,3) circle (0.05);
  \foreach \x in {7,8,9,10,11,12,13,14,15}
    \filldraw[green] (\x,2) circle (0.05);
   \draw[blue] (0,1)--(9,2)--(15,4);
   \filldraw[red] (9,2) circle (0.05);
\end{tikzpicture}
\qquad is not as good as \qquad
\begin{tikzpicture}[scale=0.35] 
  \draw[->] (0,0) -- (15,0);
  \draw[->] (0,0) -- (0,5);
  \draw[thin] (0,2)--(15,2);
  \draw (-0.1,2)--(0.1,2);
  \draw (-0.1,2) node[anchor=east] {$2$};
  \draw (6,-0.1)--(6,0.1);
  \draw (6,-0.1) node[anchor=north] {$c$};  
  \filldraw[green] (0,2) circle (0.05);
  \filldraw[green] (2,4) circle (0.05);
  \filldraw[green] (4,4) circle (0.05);
  \filldraw[green] (6,3) circle (0.05);
  \foreach \x in {7,8,9,10,11,12,13,14,15}
    \filldraw[green] (\x,2) circle (0.05);
  \draw[blue] (0,2)--(15,2);
\end{tikzpicture}
\end{center}
\end{figure}

 If $w_0 > 2$, the graph of $w$ can cross the line $y=2$ one or two times.

 If the first crossing occurs after the conductor, $w$ is not globally
 optimal because we do better by inserting a corner near the crossing
 and replacing the graph to the right by a horizontal line at height
 2. In pictures:

\begin{figure}[H] 
\begin{center}
\begin{tikzpicture}[scale=0.35] 
  \draw[->] (0,0) -- (15,0);
  \draw[->] (0,0) -- (0,5);
  \draw[thin] (0,2)--(15,2);
  \draw (-0.1,2)--(0.1,2);
  \draw (-0.1,2) node[anchor=east] {$2$};
  \draw (6,-0.1)--(6,0.1);
  \draw (6,-0.1) node[anchor=north] {$c$};  
  \filldraw[green] (0,2) circle (0.05);
  \filldraw[green] (2,4) circle (0.05);
  \filldraw[green] (4,4) circle (0.05);
  \filldraw[green] (6,3) circle (0.05);
  \foreach \x in {7,8,9,10,11,12,13,14,15}
    \filldraw[green] (\x,2) circle (0.05);
  \draw[blue] (0,3.5)--(6,2.5)--(15,1.6);
  \filldraw[red] (6,2.5) circle (0.05);
\end{tikzpicture}
\qquad is not as good as \qquad
\begin{tikzpicture}[scale=0.35] 
  \draw[->] (0,0) -- (15,0);
  \draw[->] (0,0) -- (0,5);
  \draw[thin] (0,2)--(15,2);
  \draw (-0.1,2)--(0.1,2);
  \draw (-0.1,2) node[anchor=east] {$2$};
  \draw (6,-0.1)--(6,0.1);
  \draw (6,-0.1) node[anchor=north] {$c$};  
  \filldraw[green] (0,2) circle (0.05);
  \filldraw[green] (2,4) circle (0.05);
  \filldraw[green] (4,4) circle (0.05);
  \filldraw[green] (6,3) circle (0.05);
  \foreach \x in {7,8,9,10,11,12,13,14,15}
    \filldraw[green] (\x,2) circle (0.05);
  \draw[blue] (0,3.5)--(6,2.5)--(11,2)--(15,2);
  \filldraw[red] (6,2.5) circle (0.05);
  \filldraw[red] (11,2) circle (0.05);
\end{tikzpicture}
\end{center}
\end{figure} 

 If the first crossing occurs before the conductor, and the second
 crossing after the conductor, $w$ is not globally optimal because we
 do better by inserting a corner near the first crossing and replacing
 the graph to the right by a horizontal line at height 2. In pictures:

\begin{figure}[H]
 \begin{center}
\begin{tikzpicture}[scale=0.35] 
  \draw[->] (0,0) -- (15,0);
  \draw[->] (0,0) -- (0,5);
  \draw[thin] (0,2)--(15,2);
  \draw (-0.1,2)--(0.1,2);
  \draw (-0.1,2) node[anchor=east] {$2$};
  \draw (6,-0.1)--(6,0.1);
  \draw (6,-0.1) node[anchor=north] {$c$};  
  \filldraw[green] (0,2) circle (0.05);
  \filldraw[green] (2,4) circle (0.05);
  \filldraw[green] (4,4) circle (0.05);
  \filldraw[green] (6,3) circle (0.05);
  \foreach \x in {7,8,9,10,11,12,13,14,15}
    \filldraw[green] (\x,2) circle (0.05);
  \draw[blue] (0,4)--(2,2.5)--(4,1.5)--(11,1.5)--(15,3);
  \filldraw[red] (2,2.5) circle (0.05);
  \filldraw[red] (4,1.5) circle (0.05);
  \filldraw[red] (11,1.5) circle (0.05);
\end{tikzpicture}
\qquad is not as good as \qquad
\begin{tikzpicture}[scale=0.35] 
  \draw[->] (0,0) -- (15,0);
  \draw[->] (0,0) -- (0,5);
  \draw[thin] (0,2)--(15,2);
  \draw (-0.1,2)--(0.1,2);
  \draw (-0.1,2) node[anchor=east] {$2$};
  \draw (6,-0.1)--(6,0.1);
  \draw (6,-0.1) node[anchor=north] {$c$};  
  \filldraw[green] (0,2) circle (0.05);
  \filldraw[green] (2,4) circle (0.05);
  \filldraw[green] (4,4) circle (0.05);
  \filldraw[green] (6,3) circle (0.05);
  \foreach \x in {7,8,9,10,11,12,13,14,15}
    \filldraw[green] (\x,2) circle (0.05);
  \draw[blue] (0,4)--(2,2.5)--(3,2)--(15,2);
  \filldraw[red] (2,2.5) circle (0.05);
  \filldraw[red] (3,2) circle (0.05);
\end{tikzpicture}
\end{center}
\end{figure}
\end{proof}

\section{If all the corners are at or below the conductor} \label{sec: corners below the conductor}

By Lemma \ref{lem: polynomiality of Q Pk chi psi}, the polynomial
$\chi$ has degree at most 3. From a computer search of numerical
semigroups of genus $1 \leq g \leq 14$, we see that in general, $\chi$ has degree 3, but we also found examples where the degree drops. The next lemma describes what happens when the degree of $\chi$ drops.

\begin{lemma} \label{lem: chi3 chi2 vanishing}
  Suppose that $\max(I) \leq \condIndex(\Gamma)$. Suppose $k\geq \condIndex(\Gamma)+2$, and consider the Taylor expansions of $\chi$ and $\psi$ centered at $k$.
\begin{enumerate}
\item If $\chi_3 = 0$, then $\psi_2=0$, and $\chi_2 = - \psi_1$. 
\item If $\chi_3 = \chi_2 = 0$, then $\chi_1 = 0$, and $\psi_2 = \psi_1 = \psi_0=0$. 
\end{enumerate}
\end{lemma}
\begin{proof}
  Proof of Part (1): By Lemma \ref{lem: chi3 psi2 relation}, we have $\psi_2 = -\frac{3}{2} \chi_3$. For the second equation, we use the formulas from Lemma \ref{lem: formulas for chi and psi}. Since $\chi_3 = 0$, we have
\[ \del_{k,k+1}^{k,k+1} A'(k,k-1) = - \del_{k-1,k+1}^{k,k+1} A'(k,k-1),
\]
and hence
\[
\chi_2 = (-1)^{k+1} 2 \del_{k,k+1}^{k,k+1} A'(k,k-1).
\]
Since $\psi_2 = 0$, we have
\[
\del_{k+1}^{k+1}  A'(k,k) = - \del_{k}^{k+1}  A'(k,k),
\]
and hence
\[
\psi_1 = (-1)^{k+1} (-2) \del_{k}^{k+1}  A'(k,k).
\]

Expanding $d_{k}^{k+1}  A'(k,k)$ along its bottom row yields
\[
\del_{k}^{k+1}  A'(k,k) = \del_{k,k+1}^{k-1,k+1}  A'(k,k)
\]
and the minors $d_{k,k+1}^{k-1,k+1}  A'(k,k) $ and $ d_{k,k+1}^{k,k+1} A'(k,k-1)$ are the same. 

Hence, we obtain $\chi_2 = -\psi_1$.

Proof of Part (2):  By Part (1), we have $\psi_2 = \psi_1 = 0$. First, we prove that $\psi_0 = 0$.

Similarly, from the formulas for $\psi_2$ and $\psi_1$ given above, we see that when both $\psi_2$ and $\psi_1$ vanish, we must have 
\begin{align*}
  \del_{k+1}^{k+1}  A'(k,k) &= 0\\
  \del_{k}^{k+1}  A'(k,k) &= 0.
\end{align*}

By expanding $\del_{k}^{k+1}  A'(k,k) $ along its bottom row, we obtain
\begin{align}
\del_{k,k+1}^{k-1,k+1}  A'(k,k) &= 0. \label{eqn: vanishing minor}
\end{align}

Now we study $A'(k+1,k+1)$. By expanding along its bottom row, we obtain
\begin{equation}
\det A'(k+1,k+1) = - \del_{k+2}^{k} A'(k+1,k+1) + 2 \del_{k+2}^{k+2} A'(k+1,k+1).  \label{eqn: det A'(k+1,k+1)}
\end{equation}

Expanding $\del_{k+2}^{k} A'(k+1,k+1) $ along its bottom row yields
\[
\del_{k+2}^{k} A'(k+1,k+1) = \del_{k,k+2}^{k,k+2} A'(k+1,k+1)
\]
Expanding this along its bottom row yields 
\[
\del_{k,k+2}^{k,k+2} A'(k+1,k+1) = \del_{k,k+1,k+2}^{k-1,k,k+2} A'(k+1,k+1)
\]
But we have
\[
\Del_{k,k+1,k+2}^{k-1,k,k+2} A'(k+1,k+1) = \Del_{k,k+1}^{k-1,k+1} A'(k,k)
\]
and by (\ref{eqn: vanishing minor}), we have $\del_{k,k+1}^{k-1,k+1} A'(k,k)=0$. Thus, the first term on the right hand side of equation (\ref{eqn: det A'(k+1,k+1)}) vanishes.

Now consider the second term in equation (\ref{eqn: det A'(k+1,k+1)}). Expanding along its bottom row, we have
\[
\del_{k+2}^{k+2} A'(k+1,k+1) = -\del_{k+1,k+2}^{k-1,k+2} A'(k+1,k+1) - 2 \del_{k+1,k+2}^{k,k+2} A'(k+1,k+1) 
\]
Both of these terms vanish. Indeed,
\begin{align*}
  \del_{k+1,k+2}^{k-1,k+2} A'(k+1,k+1) &= \del_{k,k+1,k+2}^{k-1,k+1,k+2} A'(k+1,k+1) \\
&= \del_{k,k+1}^{k-1,k+1} A'(k,k)\\
  &= 0,
\end{align*}
and
\begin{align*}
  \del_{k+1,k+2}^{k,k+2} A'(k+1,k+1) &= \del_{k+1}^{k+1} A'(k,k)\\
  &= 0.
\end{align*}
This completes the proof that $\psi_0 = 0$. 

Finally, we prove that $\chi_1 = 0$. From the formulas for $\chi_3$ and $\chi_2$ given in Part (1) above, we see that when both $\chi_3$ and $\chi_2$ vanish, we must have 
\begin{align*}
  \del_{k,k+1}^{k,k+1} A'(k,k-1) &= 0 \\
  \del_{k-1,k+1}^{k,k+1} A'(k,k-1) &= 0.
\end{align*}
These are the first two terms in the formula for $\chi_1$ from Lemma \ref{lem: formulas for chi and psi}. The remaining term is a multiple of $\del_{k+2}^{k+1} A'(k+1,k).$ But we have
\[
\Del_{k+2}^{k+1} A'(k+1,k) = A'(N,N)
\]
and $\det A'(N,N)=0$, since $\psi_0 = 0$. This shows that $\chi_1 =0$.

\end{proof}

\textit{Remark: } A computer search shows that for all numerical semigroups of genus $g\leq 14$, if $\chi_3 = \chi_2 = 0$, then $\chi_0 = 0$, too, so that $\chi(N) \equiv 0$ and $\psi(N) \equiv 0$. However, we do not know whether this always holds.

\begin{definition}
We write $\nu(\Gamma,I)$ for the smallest positive integer that is
greater than the real roots of the polynomials $\chi$ and $(\chi+2(N-\condIndex(\Gamma)-2) \psi)$, and greater than $\condIndex(\Gamma)+2$ and $\max(I)+2$. We abbreviate this as $\nu$ when $\Gamma$ and $I$ are clear from context. 
\end{definition}

\begin{proposition} \label{prop: persistence and optimality 1}
  Suppose that $\max(I) \leq \condIndex(\Gamma)$. If $N > \nu$, then the global optimum for the Simplified Problem does not lie on face $I$.
\end{proposition}

\begin{proof}

We will break into cases according to the degree of $\chi(N)$. We have $\deg \chi(N) \leq 3$. In the general case, when $\deg \chi(N) = 3$, we will show that for sufficiently large $N$, the last line segment in the graph corresponding to the stationary point on face $I$ crosses the line $y=2$ past the conductor; hence, by Lemma \ref{lem: crossing implies not globally optimal}, it will not be the global optimum. We will treat the cases where $\deg \chi(N) < 3$ separately.

Write $\ell = \max(I)$. Consider the last line segment $L$ in the graph. It connects the points $(\gamma_{\ell},w_{\ell})$ and $(\gamma_N, w_N)$. Let $k = \condIndex(\Gamma)+2$. Since $\ell \leq \condIndex(\Gamma)$, the point $(\gamma_k,w_k)$ lies on the line segment $L$. We will show that if $\deg \chi(N) \geq 2$, the leading terms of $2-w_{k}$ and $2-w_{N}$ have opposite signs.

We have
\begin{align*}
  2-w_N &= 2-x_{N+1} \\
        &= -\frac{1}{2}x_{N-1} \\
        &=- \frac{\chi}{2Q}.
\end{align*}
 By our notation conventions, this line segment has slope $-x_N$ and goes through the point $(\gamma_N, w_N)$. Hence, it has equation $y - x_{N+1} = -x_N( x - \gamma_N)$. We have $w_k -x_{N+1} = -x_{N}(\gamma_k-\gamma_N)$. Since $k$ and $N$ are both larger than $\condIndex(\Gamma)$, we have $\gamma_k-\gamma_N = k-N$. Thus $w_k = x_{N+1}-x_{N}(k-N)$. Then
\begin{align*}
  2-w_k & = 2 - x_{N+1} - (N-k) x_N\\
        &= -\frac{1}{2} x_{N-1} - (N-k) x_N \\
        &= -\frac{1}{2}( x_{N-1} + 2(N-k)x_N)\\
        &= -\frac{(\chi+2(N-k) \psi )}{2Q}.
\end{align*}

Choose $k$ as the center for the Taylor expansions of $\chi$ and $\psi$. Then we have
\[
2-w_k  = -\frac{\left( (\chi_3 + 2 \psi_2) (N-k)^3 + (\chi_2 + 2 \psi_1) (N-k)^2 + (\chi_1 + 2 \psi_0)(N-k) + \chi_0 \right) }{2Q}
\]

We now split into cases according to the degree of $\chi(N)$.

\textit{Case 1:} $\deg \chi = 3$. Thus, $\chi_3 \neq 0$. Then by Lemma \ref{lem: chi3 psi2 relation}, we have $\psi_2 = -\frac{3}{2} \chi_3$. Thus, the leading coefficient of the numerator of $2-w_k$ is $\chi_3 + 2 \psi_2 = -2 \chi_3$, while he leading coefficient of the numerator of $2-w_N$ is $\chi_3$. When $N > \nu$, the signs of these polynomials are determined by the signs of their leading coefficients. Hence, for all $N>\nu$, the line segment $L$ crosses the line $y=2$ after the conductor of $\Gamma$. By Lemma \ref{lem: crossing implies not globally optimal} it follows that the stationary point on face $I$ will not be the global optimum for the Simplified Problem for $N$ in this range.

\textit{Case 2:} $\deg \chi = 2$. Thus, $\chi_3 = 0$, but $\chi_2 \neq 0.$ By Lemma \ref{lem: chi3 chi2 vanishing}, we have $\psi_2 = 0$ and $\chi_2 = -\psi_1$.  

The leading coefficient of the numerator of $2-w_N$ is $\chi_2$. The leading coefficient of the numerator of $2-w_k$ is $\chi_2 + 2 \psi_1 = -\chi_2$. Since these have opposite signs, we may argue as we did in Case 1.  When $N > \nu$, the signs of these polynomials are determined by the signs of their leading coefficients. Hence, for all $N>\nu$, the line segment $L$ crosses the line $y=2$ after the conductor of $\Gamma$. By Lemma \ref{lem: crossing implies not globally optimal} it follows that the stationary point on face $I$ will not be the global optimum for the Simplified Problem for $N$ in this range.

\textit{Case 3:} $\deg \chi = 1$. By Lemma \ref{lem: chi3 chi2 vanishing}, this does not occur.

\textit{Case 4:} $\deg \chi= 0$, but $\chi_0 \neq 0$. 

By Lemma \ref{lem: chi3 chi2 vanishing}, we have $\psi(N) \equiv 0$, and $\chi(N) = \chi_0$.

Sum the KKT equations starting in row $\ell+2$, where $\ell = \max I$. We omit $x_N$, since $x_N = 0$. We obtain
\[
(\gamma_{\ell+2}-\gamma_{\ell+1}) x_{\ell+1} + 2(N-\ell)x_{N+1} = 4(N-\ell)
\]
Thus
\begin{align*}
  (\gamma_{\ell+2}-\gamma_{\ell+1}) x_{\ell+1} &= (N-\ell)(4-2x_{N+1})\\
  & = (N-\ell)(-x_{N-1}) \\
  &= \frac{(N-\ell)(-\chi_0)}{Q(N)}.
\end{align*}

Thus, this stationary point cannot be the global optimum for any $N$, since one of the KKT multipliers $x_{\ell+1}$ or $x_{N-1}$ is negative.

\textit{Case 5:} $ \chi(N) \equiv 0$. 

We will study rows $\condIndex+1,\ldots,N-1$ of the KKT matrix equation, and eventually split into four further cases.

Since $\chi\equiv 0$ and $\psi \equiv 0$, we have $x_{N-1} = x_N = 0$. The last row of the KKT matrix equation is $2x_{N+1} = 4$, so $x_{N+1} = 2$.

The penultimate row of the KKT matrix equation is $-x_{N-2} + 2x_{N-1} + 2x_{N} + 2x_{N+1} = 4$. With $x_{N-1} = x_N = 0$ and $x_{N+1} = 2$, we have $x_{N-2} = 0$.

By induction, using rows $\condIndex+2,\ldots,N+1$ of the KKT matrix equation, we may show that $x_{i}  = 0$ for $\condIndex \leq i \leq N-1$. In all these rows, the right hand side is $4$.

In row $\condIndex + 1$, on the right hand side, we have $2a_{\condIndex+1} = 2(\gamma_{\condIndex+1} - \gamma_{\condIndex-1})$. This is  because $\gamma_{\condIndex+1} - \gamma_{\condIndex-1} = (\gamma_{\condIndex+1} - \gamma_{\condIndex}) + (\gamma_{\condIndex+1} - \gamma_{\condIndex-1})$, and $\gamma_{\condIndex+1} - \gamma_{\condIndex}=1$ while $\gamma_{\condIndex} - \gamma_{\condIndex-1} >1$ by the definition of the conductor.

We split into cases that determine the left hand side of row $\condIndex + 1$ of the KKT matrix equation. 

\textit{Case 5a:} $ \chi(N) \equiv 0$ and $\ell < \condIndex-1$.

Then row $\condIndex + 1$ says
\[
 (\gamma_{\condIndex-2} - \gamma_{\condIndex-1})x_{\condIndex-1} + 4 = 2a_{\condIndex+1}. 
\]
Since  $(\gamma_{\condIndex-2} - \gamma_{\condIndex-1})<0$ and  $2a_{\condIndex+1}>4$, this implies $x_{\condIndex-1} < 0$. Thus, this stationary point cannot be the global optimum for any $N$.

\textit{Case 5b:} $ \chi(N) \equiv 0$ and $\ell = \condIndex -1$.  In this case, row $\condIndex+1$ says $4 = 2a_{\condIndex+1}>4$, a contradiction. 

\textit{Case 5c:} $ \chi(N) \equiv 0$, $\ell = \condIndex$, and $\ell-1 \not\in I$.  This case is similar to Case 5a. In this case, row $\condIndex+1$ says
\[
 (\gamma_{\condIndex-2} - \gamma_{\condIndex-1})x_{\condIndex-1} + 4 = 2a_{\condIndex+1}. 
\]
Since  $(\gamma_{\condIndex-2} - \gamma_{\condIndex-1})<0$ and  $2a_{\condIndex+1}>4$, this implies $x_{\condIndex-1} < 0$. Thus, this stationary point cannot be the global optimum for any $N$.

\textit{Case 5d:} $ \chi(N) \equiv 0$, $\ell = \condIndex$, and $\ell-1 \in I$. This case is imilar to Case 5b.  In this case, row $\condIndex+1$ says $4 = 2a_{\condIndex+1}>4$, a contradiction. 
\end{proof}

\begin{corollary} \label{cor: N1}
Let $N_1 = \max_{I : \max(I) \leq \condIndex(\Gamma)} \{ \nu(\Gamma,I) \}$. If
$N > N_1$, then the global optimum for the Simplified Problem has at least one corner  above the conductor.  
\end{corollary}

\section{If there is at least one corner above the conductor} \label{sec: corners above the conductor}

In this section, we consider the case where there is at least one
corner above the conductor.

Let $\ell = \max(I)$. We must have $N \geq \ell+1$. We will split into two cases: $N = \ell+1$, and $N > \ell+1$. There are two reasons for this case split. First, the behavior of the stationary points is different in these two cases. Second, Lemmas \ref{lem: polynomiality of Q Pk chi psi} and \ref{lem: formulas for chi and psi} require $N \geq \ell+2$.

\subsection{If $N = \ell+1$}
\begin{proposition} \label{prop: when N=ell + one}
  Let $N'$ be an integer with $N' \geq \condIndex(\Gamma)+1$. Suppose that the global optimum for the Simplified Problem for $N  = N'$ occurs on a face $I$ with $\condIndex(\Gamma) < \ell=  N'-1$, where $\ell = \max(I)$.  Then either the stationary point on face $I$ is persistent, or the stationary point on face $I' = I \cup \{\ell+1\}$ is persistent.
\end{proposition}
\begin{proof} 

  Consider the KKT matrix equation when $N = N'$. The last row corresponds to the equation $2x_{N+1} = 4$, so we have $x_{N+1}=2$. If this solution also has $x_N = 0$, then it is persistent.

  So suppose that $x_N \neq 0$. Let $I' = I \cup \{\ell+1\}$, and
  consider the KKT matrix equation when $N = N'+1 = \ell+2.$ The last
  two rows correspond to the equations $2x_N + 2x_{N+1}=4$ and
  $2x_{N+1}=4$. Hence $x_N=0$ and $x_{N+1}=2$, so we have a persistent solution. 
\end{proof}

\begin{definition}
We say that a set of corners $I$ \emph{heralds persistence} if the stationary point on face $I$ is not persistent, but the stationary point on face $I' = I \cup \{\ell+1\}$ is persistent, where $\ell = \max I$.
\end{definition}

\subsection{If $N > \ell+1$}
When there is at least one corner above the conductor, and $N >
\ell+1$, we will show that the polynomials $\chi$ and $\psi$ factor in
a specific form, and their signs are determined by their leading coefficients. It is thus easier to relate optimality and persistence in this case than it is when all the corners are below the conductor.

\begin{lemma} \label{lem: chi psi relation when corner above conductor}
  Write $\ell = \max I$. Suppose $\ell \geq \condIndex(\Gamma)$, and  $N > \ell+1$. Then
\[
\chi = - \frac{2}{3}(N-\ell-1) \psi.
\]
\end{lemma}
\begin{proof}
  Sum the KKT equations starting in row $\ell+1$, where $\ell = \max
  I$. For each $j$ satisfying $\ell+2 \leq j \leq N-2$, the sum in
  column $j$ is 0. Furthermore, since $\ell \geq \condIndex$, in column
  $\ell+1$ we have the entries $-1, 2, -1$ in rows $\ell, \ell+1,
  \ell+2$. 

  We obtain
  \[
  x_{\ell+1}  + \left( \sum_{i=0}^{N-\ell-1} 2i \right) x_N + 2(N-\ell)x_{N+1} = 4 (N-\ell)
  \]
  Rearranging and substituting $-x_{N-1} = 4-2x_{N+1}$ yields 
  \begin{equation} \label{eq: xlp1 1}
  x_{\ell+1} = (N-\ell) (-x_{N-1} - (N-\ell-1)x_N)  
  \end{equation}
  By Lemma \ref{lem: Pj} we also have
  \begin{equation}\label{eq: xlp1 2}
x_{\ell+1} = (N-\ell-1)(N-\ell) \left(  \frac{1}{2}x_{N-1} + \frac{1}{3}(N-\ell-2)x_N \right)
  \end{equation}
  Combining equations (\ref{eq: xlp1 1}) and (\ref{eq: xlp1 2}) yields
\[
(N-\ell) (-x_{N-1} - (N-\ell-1)x_N)   = (N-\ell-1)(N-\ell) \left(  \frac{1}{2}x_{N-1} + \frac{1}{3}(N-\ell-2)x_N \right)
\]
Cancelling the factor $(N-\ell)$ on both sides yields
\[
-x_{N-1}- (N-\ell-1)x_N= (N-\ell-1)  \left(  \frac{1}{2}x_{N-1} + \frac{1}{3}(N-\ell-2)x_N\right).
\]
This rearranges to
\[
  x_{N-1} = - \frac{2}{3} (N-\ell-1) x_N.
\]
Cancelling the common denominator $Q$ yields the desired result.

\end{proof}

\begin{lemma} \label{lem: factoring conditions}
A polynomial $c_2 (x-\ell-1)^2 + c_1(x-\ell-1) +c_0$ factors as $c_2 (x-\ell)(x-\ell+1)$ if and only if $c_0 = 2c_2$ and $c_1 = 3c_2$.
\end{lemma}
\begin{proof}
Left to the reader.
\end{proof}

\begin{proposition} \label{prop: factoring chi and psi}
  Write $\ell = \max I$. Suppose $\ell \geq \condIndex(\Gamma)$, and  $N > \ell+1$. Then the polynomials $\chi$ and $\psi$ factor as follows.
  \begin{eqnarray*}
   \chi&=& (N-\ell-1)(N-\ell)(N-\ell+1)\chi_{3} \\
    \psi &=& (N-\ell)(N-\ell+1)\psi_{2}
  \end{eqnarray*}
\end{proposition}
\begin{proof}
  First, we prove the claim for the polynomial $\psi$.

  Lemma \ref{lem: formulas for chi and psi} gives formulas for the Taylor expansion of $\psi$ with center $k = \ell+1$.
  We have:
  \[
   \psi  = \psi_2 (N-\ell-1)^2 + \psi_2 (N-\ell-1) + \psi_0
  \]
  where
  \begin{align*}
    \psi_2 & = (-1)^{\ell+2} \left(\del_{\ell+2}^{\ell+2}  A'(\ell+1, \ell+1) + \del_{\ell+1}^{\ell+2}  A'(\ell+1, \ell+1) \right) \\
    \psi_1 & = (-1)^{\ell+2} \left(3\del_{\ell+2}^{\ell+2}  A'(\ell+1, \ell+1) + \del_{\ell+1}^{\ell+2}  A'(\ell+1, \ell+1) \right) 
    \psi_0 &= (-1)^{\ell+2} A'(\ell+1, \ell+1).
  \end{align*}

  By Lemma \ref{lem: factoring conditions}, to achieve the desired result, it is enough to check that $\psi_0 = 2 \psi_2$ and $\psi_1 = 3\psi_2$.

  The bottom row of the matrix $A'(\ell+1, \ell+1)$ is 0, except in the last column, where it is 2. Thus expanding along the bottom row yields
  \[
\det A'(\ell+1, \ell+1) = \del_{\ell+2}^{\ell+2}  A'(\ell+1, \ell+1).
  \]
  Also, $\del_{\ell+1}^{\ell+2}  A'(\ell+1, \ell+1) = 0$ because this bottom row of this minor is 0.

  We have $\psi_0 = 2 \psi_2$ and $\psi_1 = 3\psi_2$, and hence the polynomial $\psi$ factors as claimed.

  To prove the desired claim for $\chi$, first we apply Lemma \ref{lem: chi psi relation when corner above conductor} to get
  \[
  \chi = - \frac{2}{3}(N-\ell-1) \psi = - \frac{2}{3}\psi_2(N-\ell-1) (N-\ell)(N-\ell+1).
  \]
Next, by Lemma \ref{lem: chi3 psi2 relation}, we have $\chi_3 = - \frac{2}{3} \psi_2$. This yields the desired result.

\end{proof}

\begin{proposition} \label{prop: persistence and optimality 2}
Write $\ell = \max I$. Suppose $\ell \geq \condIndex(\Gamma)$.  If the global optimum for the Simplified Problem lies on face $I$ for at least one $N$, then either this solution is persistent, or it heralds persistence.
\end{proposition}
\begin{proof}
 If $N = \ell+1$, then Proposition \ref{prop: when N=ell + one} gives the desired result.

  So suppose $N > \ell+1$, and suppose that the global optimum lies on face $I$ for this $N$. We split into two cases according to whether $\psi_2$ is 0. 

  \textit{Case 1.}  Suppose that $\psi_2 \neq 0$. We argue as we did in the proof of Proposition \ref{prop: persistence and optimality 1}. Consider the last line segment $L$ in the graph. It connects the points $(\gamma_{\ell},w_{\ell})$ and $(\gamma_N, w_N)$.

We will show that the leading terms of $2-w_{\ell}$ and $2-w_{N}$ have opposite signs. Computing as we did in the proof of Proposition \ref{prop: persistence and optimality 1}, we have
\begin{align*}
  2-w_N =- \frac{\chi}{2Q} = \frac{-\chi_3(N-\ell-1)(N-\ell)(N-\ell+1) }{2Q}
\end{align*}
and
\begin{align*}
  2-w_{\ell}&= -\frac{(\chi+2(N-\ell) \psi )}{2Q}.
\end{align*}

By Lemma 
\begin{align*}
  \chi + 2(N-\ell)\psi &= -\frac{2}{3}(N-\ell-1) \psi + 2(N-\ell)\psi \\
                       &= \frac{4}{3}(N-\ell)\psi + \frac{2}{3}\psi \\
                       &= \frac{4}{3}(N-\ell+\frac{1}{2})\psi \\
  &= \frac{4}{3}\psi_2(N-\ell+\frac{1}{2}) (N-\ell)(N-\ell+1).
\end{align*}
Thus
\[
2-w_{\ell} = \frac{-\frac{4}{3}\psi_2(N-\ell+\frac{1}{2}) (N-\ell)(N-\ell+1)}{2Q}
\]

Since $\chi_3 = - \frac{2}{3} \psi_2$, this implies that $2-w_{\ell}$ and $2-w_N$ have opposite signs.  By Lemma \ref{lem: crossing implies not globally optimal} it follows that the stationary point on face $I$ will not be the global optimum for the Simplified Problem for $N$. This contradicts the hypothesis that $I$ carries the global optimum for this $N$.

\textit{Case 2.} Suppose that $\psi_2 = 0$. Then by Lemma \ref{prop: factoring chi and psi} and Lemma \ref{lem: chi3 psi2 relation}, we have  $\psi \equiv 0$ and $\chi \equiv 0$, hence the solution is persistent.

\end{proof}

\section{Persistence for the Simplified Problem and the worst 1-PS} \label{sec: conclusion}

\subsection{Persistence for the Simplified Problem}

\begin{theorem}[Persistence of the global optimum for the Simplified Problem] \label{thm: persistence for the simplified problem}
  Let $\Gamma$ be a numerical semigroup. There exists an integer $N_0^{\simp}$ and a
set of corner indices $I^{\simp}$ (both depending on $\Gamma$) such that for all $N \geq N_0^{\simp}$, the global
optimum for the Simplified Problem for $\Gamma$ is the persistent
solution for the face of $W$ corresponding to $I^{\simp}$.
\end{theorem}
\begin{proof}

  Let $N_1 = \max_{I : \max(I) \leq \condIndex(\Gamma)} \{ \nu(\Gamma,I) \}$ as in
  Corollary \ref{cor: N1}. Consider the global optimum for the
  Simplified Problem when $N = N_1 +1$.  By the corollary, this
  solution has at least one corner above the conductor. By Proposition \ref{prop: persistence and optimality 2}, it is either persistent, or it heralds a persistent solution.

  Thus, the global optimum for the Simplified Problem when $N = N_1+2$ must be persistent. Take $N_0^{\simp} = N_1 + 2$ to complete the proof. 
\end{proof}


%
%

\subsection{Persistence for the worst 1-PS}

\begin{lemma} \label{lem: least squares}
  Consider the set of points $\{ (0,0),\ldots,(N-\ell-1,0), (N-\ell,-1)\}$. Then the least squares regression line for this set of points has slope $m$ and $y$-intercept $b'$ given by the following formulas.
  \begin{align*}
    m &= \frac{-6}{(N-\ell+1)(N-\ell+2)} \\
    b' &= \frac{2(N-\ell-1)}{(N-\ell+1)(N-\ell+2)} 
  \end{align*}
  
\end{lemma}

\begin{theorem}[Persistence of the worst 1-PS] \label{thm: persistence
    of the worst 1-PS}
  Let $\Gamma$ be a numerical semigroup.
\begin{enumerate}
\item 
There exists an integer $N_0$ and a
set of corner indices $I$ (both depending on $\Gamma$) such that for all $N \geq N_0$, the global
optimum for the Unsimplified Problem for $\Gamma$ lies on the face of $W$
corresponding to $I$, and on no smaller face.
\item Write $\ell = \max I$. For all $N \geq N_0$, the coordinates
  $w_0^*,\ldots,w_{\ell-1}^{*}$ are constant with respect to $N$, and the coordinates
  $w_{\ell}^{*},\ldots,w_{N}^{*}$ are given by the formula
\[
w_{i}^{*} = mi + b
\]
where
\begin{align*}
m &= \frac{-6}{(N-\ell+1)(N-\ell+2)} \\
b &= 2+\frac{2(N+2\ell-1)}{(N-\ell+1)(N-\ell+2)}
\end{align*}
\end{enumerate}
\end{theorem}
\begin{proof}
We modify the solution to the Simplified Problem to obtain a solution to the Unsimplified Problem.

Choose $N_0 \geq N_0^{\simp}$. Let $I^{\simp}$ be the set of corner indices giving the persistent solution to the Simplified Problem. Let $\ell = \max(I^{\simp})$. We take $I = I^{\simp} \cup \{\ell-1\}$. (Note: $\ell-1$ may already be in $I^{\simp}$, in which case $I = I^{\simp}$.) By Corollary \ref{cor: N1} we have $\ell > \condIndex(\Gamma)$.

Let $x$ be the solution to the KKT matrix equation for the Simplified Problem on face $I$. If $I = I^{\simp}$, it is the global optimum to the Simplified Problem. If $I \neq I^{\simp}$, the only difference is that we added an unused corner at index $\ell-1$. So the solution $x$ gives the global optimum to the Simplified Problem, but with the parameter $x_{\ell-1} =0$.

We will define a vector $\tilde{x}$ and then show that it is the solution to the KKT matrix equation for the Unsimplified Problem.

For the first $\ell-2$ coordinates of $\tilde{x}$, we use the solution $x$.

Set $\tilde{x}_N$ and $\tilde{x}_{N+1}$ to give the least squares regression line through the points $\{(\gamma_\ell,2),\ldots,(\gamma_{N-1},2),(\gamma_N,1)\}$. (We can translate the least squares regression line described in Lemma \ref{lem: least squares} by $(\gamma_{\ell},2)$ to obtain the desired line.) Recursively solve rows $\ell+1,\ldots,N+1$ of the KKT matrix for the Unsimplified Problem to obtain  $\tilde{x}_{\ell+1},\ldots,\tilde{x}_{N-1}$. 

Finally, we modify $x_{\ell-1}$ and $x_{\ell}$ to account for the
changes in the slopes of the last two line segments.

Explicitly, we have
\begin{equation} \label{eqn: xtilde formulas}
  \tilde{x}_i = \left\{
    \begin{array}{ll}
x_i, & 1 \leq i \leq \ell-2 \\
x_{\ell-1} + b', & i = \ell-1 \\
x_{\ell} + m - b', & i = \ell  \\
  \frac{2(i-\ell+1)(i-\ell)(N-i)}{(N-\ell+1)(N-\ell+2)}, & 
   \ell+1 \leq i \leq N-1 \\
-m, & i = N\\
m(N-\ell)+2+b', & i = N+1
\end{array}
  \right.
\end{equation}
Here $m$ and $b'$ are the quantities given in Lemma \ref{lem: least squares}.

To finish the proof of the first statement, we need to show that for $N$ sufficiently large, the following three conditions are satisfied.

\begin{enumerate}
\item[(i).] $\tilde{x}$ satisfies the KKT matrix equation for the Unsimplified Problem
\item[(ii).] $\tilde{x}_i > 0$ for $i \in I$
\item[(iii).] $\tilde{x}_i \geq 0$ for $i \not\in I$ and $i\leq N-1$. 
\end{enumerate}

\textit{Proof of (i)}.

The KKT matrix equations for the Simplified Problem and Unsimplified Problem are the same except for the last coordinate of the right hand side. Thus, when $i\leq N$, the equation $A\tilde{x} = 2a$ in row $i$ is equivalent to $A(\tilde{x}-x)=0$ in row $i$.

First, consider rows 1 through $\ell$ of the KKT matrix equation. These rows are zero in columns $\ell+1,\ldots, N-1$.  Furthermore, $(\tilde{x} - x)_i = 0$ in coordinates $1 \leq i \leq \ell-2$.  So it's enough to show that
\[
A_{i,\ell-1} (\tilde{x}_{\ell-1}-x_{\ell-1})+A_{i,\ell} (\tilde{x}_{\ell}-x_{\ell}) + A_{i,N} (\tilde{x}_{N}-x_{N})  + A_{i,N+1} (\tilde{x}_{N+1}-x_{N+1}) = 0.   
\]
When $i \leq \ell-1$, substituting the formulas for $A_{i,j}$ from Proposition \ref{prop: KKT matrix entries} and the formulas for $\tilde{x}$ from (\ref{eqn: xtilde formulas}) yields

\[
2(\gamma_{\ell-1} - \gamma_{i-1}) (b') +2(\gamma_{\ell} - \gamma_{i-1}) (m-b')  +2(\gamma_{N} - \gamma_{i-1}) (-m) + 2 (m(N-\ell)+b') = 0.   
\]

This rearranges to 
\[
[\gamma_{\ell-1} - \gamma_{\ell} + 1](2b') + [(\gamma_{\ell} -  \gamma_{N} + N-\ell](2m) = 0.
\]

But we have $\gamma_{\ell} - \gamma_{\ell-1}= 1$ and $\gamma_N - \gamma_\ell = N-\ell$ because $\ell > \condIndex(\Gamma)$. This gives the desired result.

A similar calculation yields the desired result when $i=\ell$. 

For rows $\ell+1$ through $N+1$, we need to verify the following equations.

For row $\ell+1$: 
\[
-\tilde{x}_{\ell+1} + 2(N-\ell)\tilde{x}_N + 2 \tilde{x}_{N+1} = 4.
\]
For row $\ell+2$:
\[
2\tilde{x}_{\ell+1} -\tilde{x}_{\ell+2} + 2(N-\ell-1)\tilde{x}_N + 2 \tilde{x}_{N+1} = 4.
\]
For $\ell+3 \leq i \leq N-1$:
\[
-\tilde{x}_{i-2}  +2\tilde{x}_{i-1} -\tilde{x}_{i} + 2(N-i+1)\tilde{x}_N + 2 \tilde{x}_{N+1} = 4.
\]
For row $N$:
\[
-\tilde{x}_{N-2}  +2\tilde{x}_{N-1}  + 2\tilde{x}_N + 2 \tilde{x}_{N+1} = 4.
\]
For row $N+1$:
\[
-\tilde{x}_{N-1}  + 2 \tilde{x}_{N+1} = 2.
\]

We substitute the formulas for $\tilde{x}_i$ given in (\ref{eqn: xtilde formulas}) and verify these identities of rational functions.

\textit{Proof of (ii)}. Let $i \in I$.

If $i \leq \ell-2$, then we have $\tilde{x}_i > 0$, since $\tilde{x}_i = x_i$ in this range, and $x$ is the global optimum for the Simplified Problem.

If $i = \ell-1$, we have $\tilde{x}_{\ell-1} = x_{\ell-1} + b' > 0$, since $x_{\ell-1} \geq 0$ and $b'>0$.

If $i = \ell$, we explain how to choose $N_0$ sufficiently large to ensure that $\tilde{x}_{\ell} > 0$. 

We have $\tilde{x}_{\ell} = x_{\ell} + m - b'$. We want to choose $N$ sufficiently large that
\[
x_{\ell} + m - b' > 0.
\]
Substituting the formulas for $m$ and $b'$ yields
\[
x_{\ell}     + \frac{-6}{(N-\ell+1)(N-\ell+2)} - \frac{2(N-\ell-1)}{(N-\ell+1)(N-\ell+2)} >0.
\]
Clearing denominators, we have
\[
x_{\ell} (N-\ell+1)(N-\ell+2)    -6 - 2(N-\ell-1) >0.
\]
This rearranges to
\begin{equation} \label{eqn: N minus ell quadratic}
x_{\ell} (N-\ell)^2 + (3 x_{\ell}-2)(N-\ell) + 2x_{\ell}-4 > 0.  
\end{equation}
The roots of this quadratic polynomial are
\[
N -\ell = \frac{-(3 x_{\ell}-2) \pm \sqrt{ (3 x_{\ell}-2)^2-4 x_{\ell} (2x_{\ell}-4 )} }{2 x_{\ell}}
\]
The square root simplifies as follows.
\[
\sqrt{ (3 x_{\ell}-2)^2-4 x_{\ell} (2x_{\ell}-4 )}  =  x_{\ell}+2.
\]
Thus, the roots of the quadratic polynomial in (\ref{eqn: N minus ell
  quadratic}) are $\frac{-2x_{\ell} + 4}{2x_\ell}$ and $\frac{-4x_{\ell}}{2 x_\ell}$. Since $x_{\ell} \geq 0$, the larger root is given by $-2x_{\ell} +4$.

Thus, to ensure $\tilde{x}_{\ell} > 0$, it suffices to take
\[
N - \ell > \frac{-2x_{\ell} +4}{2x_\ell}.
\]
This rearranges to
\[
N > \ell + \frac{2}{x_{\ell}} - 1.
\]

\textit{Proof of (iii)}. Since $x_i \geq 0$ for all $i$, it follows that $\tilde{x} \geq 0$ for $i=1,\ldots,\ell-2$. We already checked in Part (ii) that $\tilde{x}_{\ell-1}\geq 0$ and $\tilde{x}_{\ell}\geq 0$. For $\ell+1\leq i \leq N+1$ it is clear from the formulas that $\tilde{x} \geq 0$.

To complete the proof of the first statement, take $N_0 = \max\{
N_0^{\simp}, \lfloor \ell + \frac{2}{x_{\ell}} \rfloor  \}.$

For the second statement: by our parametrization of the cone $W$ (see
Definition \ref{def: KKT matrix equation}), a
solution $\tilde{x}$ to the KKT matrix equation corresponds to the point $w
\in W$ given by 
\begin{equation} \label{eqn: w equation}
  w = \left(\sum_{i \in I} \tilde{x}_i F_{\gamma_i}(\gamma) \right)+ \tilde{x}_N L_{\gamma_N}(\gamma) +
  \tilde{x}_{N+1} L_1(\gamma). 
\end{equation}





If $k \geq \ell$, then the $k^{th}$ coordinate of
$F_{\gamma_i}(\gamma)$ is 0 for all $i \in I$, so
\begin{align*}
w_k &= [\tilde{x}_N L_{\gamma_N}(\gamma) + \tilde{x}_{N+1} L_1(\gamma)]_k \\
&= \tilde{x}_N (\gamma_N - \gamma_k) + \tilde{x}_{N+1} \cdot 1 \\
&= -m (N-k) + m(N-\ell) + 2 + b' \\
&= mk + (-m\ell + 2 + b') \\
&= mk +b.
\end{align*}

If $k = \ell-1$, then the  the $k^{th}$ coordinate of
$F_{\gamma_i}(\gamma)$ is 0 for all $i \in I$ except $i = \ell$, so
\begin{align*}
  w_{\ell-1}  &= [\tilde{x}_{\ell} L_{\gamma_\ell}(\gamma)+\tilde{x}_N L_{\gamma_N}(\gamma)+ \tilde{x}_{N+1} L_1(\gamma)]_{\ell-1} \\
  &= \tilde{x}_{\ell}(\gamma_\ell - \gamma_{\ell-1}) + m(\ell-1) + (-m\ell + 2 + b') \\
  &= x_{\ell}+2.
\end{align*}
Then $w_{\ell-1} $ is constant with respect to $N$, since $x_{\ell}$
is constant with respect to $N$.

If $k \leq \ell-2$, then 
\begin{align*}
w_{k}  &= \left[\left(\sum_{i \in I} \tilde{x}_i F_{\gamma_i}(\gamma)\right)+\tilde{x}_N L_{\gamma_N}(\gamma)+ \tilde{x}_{N+1}\right]_k \\
&= \left[\left(\sum_{\substack{i \in I\\i \leq \ell-2}} \tilde{x}_i F_{\gamma_i}(\gamma) \right)+ \tilde{x}_{\ell-1} F_{\gamma_{\ell-1}}(\gamma) + \tilde{x}_{\ell}F_{\gamma_{\ell}}(\gamma) +\tilde{x}_N L_{\gamma_N}(\gamma)+ \tilde{x}_{N+1} L_1(\gamma)\right]_k.
\end{align*}
The sum 
\[
\sum_{\substack{i \in I\\i \leq \ell-2}} \tilde{x}_i F_{\gamma_i}(\gamma) 
\]
is constant with respect to $N$, since $\tilde{x}_i = x_i$ in this range, and both $x_i$ and $F_{\gamma_i}(\gamma) $ are constant with respect to $N$.

It remains to show that the sum of the last four terms is also constant with respect to $N$. We have
\begin{align*}
  \lefteqn{ [\tilde{x}_{\ell-1} F_{\gamma_{\ell-1}}(\gamma) + \tilde{x}_{\ell}F_{\gamma_{\ell}}(\gamma) +\tilde{x}_N L_{\gamma_N}(\gamma)+ \tilde{x}_{N+1} L_1(\gamma)]_k} \\
  &= (x_{\ell-1}+b') (\gamma_{\ell-1} - \gamma_k) +  (x_{\ell}+m-b')  (\gamma_{\ell} - \gamma_k) -m (\gamma_N - \gamma_k) + m(N-\ell) + 2 + b'\\
  &= x_{\ell-1}(\gamma_{\ell-1} - \gamma_k) + x_{\ell}(\gamma_{\ell} - \gamma_k).
\end{align*}
Since $\gamma_k$, $\gamma_\ell$, $x_{\ell-1}$, and $x_{\ell}$ are
constant with respect to $N$, this expression is constant with respect to $N$. This completes the proof of the second statement.
\end{proof}

\textit{Remark.} It is possible to give a more conceptual proof of the second statement in Theorem \ref{thm: persistence
  of the worst 1-PS}.  When the corner set ends in two consecutive corners $\{\ell-1,\ell\}$, to minimize $\| w - a\|^2$, we may minimize $\sum_{i=0}^{\ell-1} (w_i - a_i)^2$ and $\sum_{i=\ell}^{N} (w_i - a_i)^2$ separately. Then the first sum is independent of $N$, and is the same for the Simplified and Unsimplified Problems; and the least squares regression line minimizes the second sum. 

\section{Examples} \label{sec: examples}

In Table \ref{tab: Matlab/Octave results}, for five different numerical semigroups, we describe the corner set $I$ giving the worst 1-PS for each $N$. Here is how we generated this table.
\begin{enumerate}
\item For each $N$ starting at $\condIndex(\Gamma)+2$, we compute the face $I^{\simp}(N)$ giving the optimal solution to the Simplified Problem. We increase $N$ by one until we find the persistent solution to the Simplified Problem. 
\item We use Theorem \ref{thm: persistence
    of the worst 1-PS} to compute the persistent solution to the Unsimplified Problem.
\item For each $N$ in $[\condIndex(\Gamma)+2,N_0]$, we compute the face $I(N)$ giving the optimal solution to the Unsimplified Problem.
\end{enumerate}

\begin{table}[H]
  \caption{The worst 1-PS for five singularities}
  \label{tab: Matlab/Octave results}  
  \begin{center}
    \begin{tabular}{llll}
    Simple cusp & $\langle 2, 3\rangle$ & $I$ & $N$ \\
    && $\emptyset$ & $4 \leq N \leq 15 $\\
    &&$\{4\}$ & $16 \leq N \leq 28 $\\
    &&$\{4,5\}$ & $29 \leq N \leq 74 $\\
    &&$\{5\}$ & $75 \leq N \leq 145$\\
    &&$\{5,6\}$ & $146 \leq N $\\
    Rhamphoid cusp & $\langle 2, 5\rangle$ & $I$ & $N$ \\
    && $\emptyset$ & $4 \leq N \leq 11 $\\
    &&$\{5\}$ & $12 \leq N \leq 18 $\\
    &&$\{5,6\}$ & $19 \leq N \leq 69 $\\
    &&$\{6\}$ & $70 \leq N \leq 117$\\
    &&$\{6,7\}$ & $118 \leq N $\\
    & $\langle 4,9\rangle$ & $I$ & $N$ \\
    && $\emptyset$ & $4 \leq N \leq 13 $\\
    && $\{7\}$ & $14 \leq N \leq 18 $\\
    &&$\{7,8\}$ & $N=19 $\\
    &&$\{7,8,10\}$ & $N=20 $\\
    &&$\{7,10\}$ & $N=21$\\
    &&$\{7,10,14\}$ & $22 \leq N \leq 23 $\\
    &&$\{7,10,14,15\}$ & $24 \leq N \leq 27 $\\
    &&$\{7,10,15\}$ & $28 \leq N \leq 45 $\\
    &&$\{7,8,10,15,16\}$ & $46 \leq N $\\
    &  $\langle 5,7\rangle$ &    $I$ & $N$ \\
    && $\{9\}$ & $13 \leq N \leq 15 $\\
    &&$\{6, 9\}$ & $N=16$\\
    &&$\{9\}$ & $17 \leq N \leq 23 $\\
    &&$\{9,15\}$ & $24 \leq N \leq 34$\\
    &&$\{9,15,16\}$ & $35 \leq N \leq 83 $\\
    &&$\{9,16\}$ & $84 \leq N \leq 127 $\\
    &&$\{9,16,17\}$ & $128 \leq N $\\ 
    &  $\langle 8,13\rangle$ &             $I$ & $N$ \\
    &&$\{6, 11, 19, 26\}$ & $43 \leq N \leq 55 $\\
    &&$\{6, 11, 26\}$ & $56 \leq N \leq 58 $\\
    &&$\{6, 11, 26, 37\}$ & $59 \leq N \leq 60 $\\
    &&$\{6, 11, 26, 38\}$ & $61 \leq N \leq 63 $\\
    &&$\{6, 11, 26, 38, 47\}$ & $64 \leq N \leq 67 $\\
    &&$\{6, 11, 26, 38, 47, 48\}$ & $68 \leq N \leq 70 $\\
    &&$\{6, 11, 26, 38, 48\}$ & $71 \leq N \leq 85 $\\
    &&$\{6, 11, 26, 38, 48, 49\}$ & $86 \leq N \leq 112 $\\
    &&$\{6, 11, 26, 38, 49\}$ & $113 \leq N \leq 139 $\\
    &&$\{6, 11, 26, 38, 49,50\}$ & $140 \leq N $                   
  \end{tabular}
\end{center}
\end{table}

\subsection{The worst 1-PS for cusps} 
We describe the persistent worst 1-PS for a cusp $y^2
= x^{2r+1}$ of order $r$.

\begin{definition}
  For any integer $r \geq 1$, we define the polynomial
  \begin{equation}
  f(r,x):= (4r-2)x^3+(6r-6)x^2-(12r^3+6r^2+4r+4)x-(30r^3+18r^2).   
\end{equation}
\end{definition}

\begin{lemma}
For any fixed value of $r\geq 1$, the polynomial $f(r,x)$ has exactly one
positive real root.
\end{lemma}
\begin{proof}
We can check this directly for $r=1$. When $r>1$, the first two coefficients are positive and the last two
coefficients are negative, so by Descartes' Rule of Signs, $f(r,x)$ has
at most one positive real root. Since $f(r,0)<0$ and $\lim_{x
  \rightarrow \infty} f(r,x) > 0$, the polynomial $f(r,x)$ has
exactly one positive real root.
\end{proof}

\begin{definition}
We define $\alpha(r)$ to be the positive real root of $f(r,x)$.
\end{definition}

\begin{proposition} \label{prop: worst 1ps for cusps}
  Let $r$ be a positive integer. Let $j = \lceil \alpha(r) \rceil$,
  and let
\[
N_0 = j + \frac{j^4+4 j^3+(6 r^2+6 r+5) j^2+(12 r^2+12 r+2) j-3 r^4-6 r^3+3 r^2+6 r}{(-2 r+1) j^3+3 r j^2+(6 r^3+3 r^2+2 r-1) j+9 r^3+6 r^2-3 r}
\]
Then for all $N \geq N_0$, the weights of the worst 1-PS for
a cusp of order $r$ lies on the face of $W$
corresponding to $I = \{j,j+1\}$, and on no smaller face.
\end{proposition}

  For a full proof, see Appendix B.

The values of $j$ and $N_0$ for some small values of $r$ are given in
the Table \ref{tab: higher order cusps j and N0}.

\begin{table}[H]
  \caption{$j$ and $N_0$ versus $r$}
  \label{tab: higher order cusps j and N0}  
  \begin{center}
  \begin{tabular}{rrr}
    $r$ & $j$ & $N_0$ \\
    1 & 5 & 146 \\
    2 & 6 & 118 \\
    3 & 7 & 30 \\
    4 & 9 & 53 \\
    5 & 11 & 174 \\
    6 & 12 & 37 \\
    7 & 14 & 55 \\
    8 & 16 & 107 \\
    9 & 18 & 8369 \\
    10 & 19 & 58 \\
    11 & 21 & 88 \\    
    12 & 23 & 200 \\
    13 & 24 & 61 \\
    14 & 26 & 82\\
    15 & 28 & 131
  \end{tabular}
\end{center}
\end{table}

The value $N_0=8369$ for $r=9$ appears surprisingly large compared to the
other values in the table. For the
semigroup $\langle 2, 19 \rangle$, the Simplified Problem has 
persistent solution $I^{\simp} = \{18,19\}$ for all $N \geq
N_0^{\simp} = 19$. By the
proof of Theorem \ref{thm: persistence
    of the worst 1-PS}, $N_0$ is the smallest integer strictly larger
  than $\ell +  \frac{2}{x_\ell}-1$. We have $x_{19} = \frac{1}{4175}$ when $N = 19$, and this
small denominator leads to the large $N_0$ observed in the table.

We can nearly give a  closed formula for $\lceil
    \alpha(r) \rceil$, in the sense that we can identify this quantity as one of two consecutive integers.
  
  \begin{proposition} \label{prop: alpha(r) has two possible values}
    For every positive integer $r$, we have
    \[
  \lceil \alpha(r) \rceil \in \left\{  \left\lceil \sqrt{3}r+\frac{\sqrt{3}+1}{2} \right\rceil,  \left\lceil \sqrt{3}r+\frac{\sqrt{3}+1}{2} \right\rceil + 1\right\}
    \]
  \end{proposition}

Experimentally, it seems that $\lceil \alpha(r) \rceil$ is almost
always the smaller of these two possible values. This is true for all
$0 \leq r\leq 10^{8}$ except $r = 0,1,9$.

\section{Future work}
We conclude by briefly discussing a few possibilities for future work. Perhaps the most natural question is whether one can prove similar persistence results for the worst 1-PS's of other curves: in particular those with more complicated singularities, and those of higher geometric genus. In the latter case, when the singularities are attached to an otherwise stable curve, it seems reasonable to expect that the worst 1-PS will depend only on the singularities, in some appropriate sense. One could also ask about toric surface singularities. This question has yet to be investigated properly, but the preliminary evidence of a few example computations suggests that surfaces will not exhibit persistence in the same way as curves. 

Another direction is the interpretation of the results of this paper:
what does the persistent worst 1-PS \emph{mean}? We have explanations
for why they are as they are, and for the time to persistence, but so far only from a constrained optimisation point of view. An interesting question is whether there is some explanation purely in terms of algebraic geometry. In particular, is there some combinatorial algorithm that computes the persistent worst 1-PS purely from the singularity data?

Finally, as mentioned in the introduction, we intend in future work to combine our results here with tools from Non-reductive GIT, to construct new moduli spaces of singular curves. 

\appendix 

\section{Proof of Lemma \ref{lem: formulas for chi and psi}}
\subsection{Beginning the proof}
In this appendix we give a proof of the following result. 
\begin{proposition} \label{prop: KKT solutions are rational functions}
	Suppose $N \geq \max \{ \condIndex(\Gamma)+2,\max(I) +2 \}$. Let $x$ be the solution of the KKT matrix equation for face $I$. Then each coordinate of $x$ is a rational function in $N$.
\end{proposition}

We also obtain explicit formulas
for the coefficients of two polynomials $\chi$ and $\psi$ that figure prominently in the proof. Together, these results prove Lemma \ref{lem: formulas for chi and psi}.

Proposition \ref{prop: KKT solutions are rational functions} follows from the following lemma.

\begin{lemma} \label{lem: app polynomiality of Q Pk chi psi}
	Suppose    $N \geq \max \{ \condIndex(\Gamma)+2,\max(I) +2 \}$.
	
	\begin{enumerate}
		\item  The functions $Q(N)$, $\chi(N)$, $\psi(N)$, and $\omega(N)$ are polynomials with the following degree bounds. 
		\begin{enumerate}
			\item $\deg Q = 4$
			\item $\deg \chi \leq 3$
			\item $\deg \psi \leq 2$
			\item $\deg \omega = 4$.
		\end{enumerate}
		\item  If  $N \geq \max \{ \condIndex(\Gamma)+2,\max(I) +2 \}$ and $N \geq j+1$, the function $P_j(N)$ is a polynomial of degree at most 4.
	\end{enumerate}
	
\end{lemma}

A full proof is given in the following sections. However, we follow
the same outline to prove each part, and we describe this outline
now. As $N$ grows, the matrices $A(N)$ and $A'(N,j)$ grow in a way that is easy to describe. We can write recurrences for the matrices, then show that
the appropriate difference equations vanish to establish that each of
these functions is a polynomial in $N$ with the degree bounds claimed. Then,
we use a second difference equation to obtain the leading
coefficients.

\subsection{$Q$ and $P_j$ are polynomials}

\begin{definition}
	For each $n \geq 3$, define
	\[
	\Phi_n : \operatorname{Mat}_{n\times n} \rightarrow \operatorname{Mat}_{(n+1)\times (n+1)}
	\]
	as follows.
	\begin{itemize}
		\item Columns 1 through $(n-2)$ in $\Phi_n(M)$ are the same as in $M$, extended by 0 at the bottom.
		\item Column $n-1$ in $\Phi_n(M)$ is 0 except the last three entries,
		which are $-1,2,-1$.
		\item Column $n$ in $\Phi_n(M)$ is the sum of columns $n-1$ and $n$ in $M$, extended by 0 at the bottom.
		\item Column $n+1$ in $\Phi_n(M)$ is column $n$ from $M$, extended by 2 at the bottom.
	\end{itemize}
	
	We write $\Phi$ for the collection of maps $\{ \Phi_n\}$. We refer to
	$\Phi$ as a recurrence, and frequently omit the subscript.
\end{definition}

The application to our problem is as follows.
\begin{lemma}
	Suppose    $N \geq \max \{ \condIndex(\Gamma)+2,\max(I) +2 \}$.
	\begin{enumerate}
		\item The matrices $A(N)$ satisfy the recurrence $\Phi$. That is,
		$A(N+1) = \Phi(A(N))$.
		\item Fix $j$ and suppose $N\geq j+1$. Then the matrices $A'(N,j)$ satisfy the recurrence $\Phi$. That is,
		$A'(N+1,j) = \Phi(A'(N,j))$.
	\end{enumerate}
\end{lemma}

\begin{lemma} \label{lem: Q is a polynomial of degree 4}
	Let $\{M(N)\}$ be a sequence of $N \times N$ matrices for $N \geq N_0$
	such that $M(N+1) = \Phi(M(N))$ for all $N \geq N_0$. Then $
	h(N) = (-1)^{N+1} \det M(N+1) $ is given by a polynomial of degree at most 4.
\end{lemma}

We need a little more notation for the proof. 

\begin{definition}
	Let $I$ and $J$ be subsets of the row and column indices of a matrix $M$,
	respectively.
	
	We write $\Del_I^J M$ for the matrix obtained by
	deleting the rows whose indices belong to $I$, and deleting the
	columns whose indices belong to $J$.
	
	When $\Del_I^J M$ is a square matrix, we write  $\del_I^J M = \det \Del_I^J M$. 
\end{definition}

\begin{definition}
	Let $M(N)$ be an $N \times N$ matrix.
	
	We write $\aug(c,M(N))$ for
	the matrix obtained by adding $c$ times column $N$ to column $N-1$.
	
	We write $\red M(N) $ for
	the matrix obtained by subtracting column $N$ from column $N-1$.
\end{definition}

\begin{proof}[Proof of Lemma \ref{lem: Q is a polynomial of degree 4}]
	We show that
	the following fifth-order difference equation vanishes for any integer $k \geq
	N_0$.
	\begin{align}
		h(k+5) - 5h(k+4)+10h(k+3) -10h(k+2)+5h(k+1)-h(k) &= 0.
	\end{align}

	We define
	\begin{align*}
		\Delta_5 :=  h(k+5) - 5h(k+4)+10h(k+3) -10h(k+2)+5h(k+1)-h(k)
	\end{align*}
	Then our goal is to show that $\Delta_5=0$.
	
	Substituting $h(N) = (-1)^{N+1} \det M(N+1) $
	yields the following expression.
	\begin{equation} \label{eqn: Delta version 2}
		(-1)^{k} \Delta_5 = \det M(k+6) + 5\det M(k+5)+10\det M(k+4) +10\det M(k+3)+5\det
		M(k+2)+ \det M(k+1)
	\end{equation}
	
	Next, we find identities that will allow us to write each $\det M(k+i)$ in
	terms of the determinants of $M(k+6)$ and its minors.

	We start by relating $\det M(k+5)$ and $\det M(k+6)$. Expanding $\det
	M(k+6)$ along its bottom row yields 
	\[
	\det M(k+6) = -\del_{k+6}^{k+4} M(k+6)  + 2 \del_{k+6}^{k+6} M(k+6).
	\]
	But 
	\[
	\Del_{k+6}^{k+4} M(k+6)=\aug(1,M(k+5))
	\]
	so we obtain 
	\[
	\det M(k+6) = -\det M(k+5)  + 2 \del_{k+6}^{k+6} M(k+6).
	\]
	which rearranges to 
	\begin{equation} \label{eqn: Mkp4}
		\det M(k+5) = -\det M(k+6)  + 2 \del_{k+6}^{k+6} M(k+6).
	\end{equation}

	Next, we relate $\det M(k+4)$ and $\det M(k+6)$. Expanding 
	$\aug(1,M(k+5))$ along its bottom row yields
	\[
	\det(\aug(1,M(k+5))) = -\del_{k+5}^{k+3} \aug(1,M(k+5)) -2
	\del_{k+5}^{k+4} \aug(1,M(k+5))+2 \del_{k+5}^{k+5} \aug(1,M(k+5)).
	\]
	But 
	\begin{align*}
		\Del_{k+5}^{k+3} \aug(1,M(k+5)) &= M(k+4) \\
		\Del_{k+5}^{k+4} \aug(1,M(k+5)) &= \Del_{k+5,k+6}^{k+4,k+5} M(k+6)
		\\
		\Del_{k+5}^{k+5} \aug(1,M(k+5)) &= \Del_{k+5,k+6}^{k+4,k+6} M(k+6) 
	\end{align*}
	Combining these identities and rearranging yields
	\begin{equation} \label{eqn: Mkp3}
		\det M(k+4)) = -\det M(k+5)-2 \del_{k+5,k+6}^{k+4,k+5} M(k+6) +2
		\del_{k+5,k+6}^{k+4,k+6} M(k+6).
	\end{equation}

	In a similar fashion, we obtain the following identities.
	\begin{align}
		\det M(k+3) &= -\det M(k+4)  -4 \del_{k+4,k+5,k+6}^{k+3,k+4,k+5} M(k+6)+2 \del_{k+4,k+5,k+6}^{k+3,k+4,k+6} M(k+6) \label{eqn: Mkp2} \\
		\det M(k+2) &= -\det M(k+3)  -6\del_{k+3,k+4,k+5,k+6}^{k+2,k+3,k+4,k+5} M(k+6) +2 \del_{k+3,k+4,k+5,k+6}^{k+2,k+3,k+4,k+6} M(k+6)  \label{eqn: Mkp1} \\
		\det M(k)  &= -\det M(k+2)  -8\del_{k+2,k+3,k+4,k+5,k+6}^{k+1,k+2,k+3,k+4,k+5} M(k+6)+2\del_{k+2,k+3,k+4,k+5,k+6}^{k+1,k+2,k+3,k+4,k+6} M(k+6).  \label{eqn: Mkp0} 
	\end{align}
	
	Substituting the identities (\ref{eqn: Mkp4}), (\ref{eqn: Mkp3}), (\ref{eqn: Mkp2}), (\ref{eqn: Mkp1}), and (\ref{eqn: Mkp0}) into the expression (\ref{eqn: Delta version 2}) for $ (-1)^{k} \Delta_5$ yields
	\begin{align}
		(-1)^{k} \Delta_5 &= -4 \del_{k+2,k+3,k+4,k+5,k+6}^{k+1,k+2,k+3,k+4,k+5}
		M(k+6))+\del_{k+2,k+3,k+4,k+5,k+6}^{k+1,k+2,k+3,k+4,k+6} M(k+6)
		\nonumber \\
		& -12\del_{k+3,k+4,k+5,k+6}^{k+2,k+3,k+4,k+5} M(k+6) 
		+4 \del_{k+3,k+4,k+5,k+6}^{k+2,k+3,k+4,k+6} M(k+6)-12\del_{k+4,k+5,k+6}^{k+3,k+4,k+5} M(k+6) \nonumber \\
		& +6 \del_{k+4,k+5,k+6}^{k+3,k+4,k+6} M(k+6) 
		-4\del_{k+5,k+6}^{k+4,k+5} M(k+6)+4\del_{k+5,k+6}^{k+4,k+6} M(k+6) + \del_{k+6}^{k+6} M(k+6))
	\end{align}
	
	Note that this expression only involves the determinant of one matrix, $M(k+6)$, and its minors.

	Next, we successively expand the terms with fewer than five deletions.
	
	We begin with the term with one deletion. Expanding $\Del_{k+6}^{k+6} M(k+6)$ along tis bottom row yields
	\[
	\del_{k+6}^{k+6} M(k+6) = - \del_{k+5,k+6}^{k+3,k+6} M(k+6)-2\del_{k+5,k+6}^{k+4,k+6} M(k+6)+2\del_{k+5,k+6}^{k+5,k+6} M(k+6).
	\]
	We substitute this into the previous expression for $ (-1)^{k} \Delta_5$ and simplify to obtain 
	\begin{align}
		(-1)^{k} \Delta_5& =  -4 \del_{k+2,k+3,k+4,k+5,k+6}^{k+1,k+2,k+3,k+4,k+5} M(k+6))+\del_{k+2,k+3,k+4,k+5,k+6}^{k+1,k+2,k+3,k+4,k+6} M(k+6) \nonumber \\
		&-12\del_{k+3,k+4,k+5,k+6}^{k+2,k+3,k+4,k+5} M(k+6) 
		+4 \del_{k+3,k+4,k+5,k+6}^{k+2,k+3,k+4,k+6} M(k+6)-12\del_{k+4,k+5,k+6}^{k+3,k+4,k+5} M(k+6) \nonumber \\
		&+6 \del_{k+4,k+5,k+6}^{k+3,k+4,k+6} M(k+6) 
		-4\del_{k+5,k+6}^{k+4,k+5} M(k+6)+2\del_{k+5,k+6}^{k+4,k+6} M(k+6) \nonumber\\
		&-\del_{k+5,k+6}^{k+3,k+6} M(k+6)+2\del_{k+5,k+6}^{k+5,k+6} M(k+6).
	\end{align}
	
	Next, we expand the terms with two deletions along their bottom rows.
	\begin{align*}
		\del_{k+5,k+6}^{k+4,k+5} M(k+6) &= -\del_{k+4,k+5,k+6}^{k+2,k+4,k+5} M(k+6)-2 \del_{k+4,k+5,k+6}^{k+3,k+4,k+5} M(k+6)+2\del_{k+4,k+5,k+6}^{k+4,k+5,k+6} M(k+6) \\
		\del_{k+5,k+6}^{k+4,k+6} M(k+6) &=-\del_{k+4,k+5,k+6}^{k+2,k+4,k+6} M(k+6)-2\del_{k+4,k+5,k+6}^{k+3,k+4,k+6} M(k+6)+4 \del_{k+4,k+5,k+6}^{k+4,k+5,k+6} M(k+6) \\
		\del_{k+5,k+6}^{k+3,k+6} M(k+6) &=-\del_{k+4,k+5,k+6}^{k+2,k+3,k+6} M(k+6)+\del_{k+4,k+5,k+6}^{k+3,k+4,k+6} M(k+6)+4 \del_{k+4,k+5,k+6}^{k+3,k+5,k+6} M(k+6)
	\end{align*}
	
	We substitute these identities into the previous expression for $\Delta_5$ and simplify to obtain
	\begin{align}
		(-1)^{k} \Delta_5 &=  -4 \del_{k+2,k+3,k+4,k+5,k+6}^{k+1,k+2,k+3,k+4,k+5} M(k+6))+\del_{k+2,k+3,k+4,k+5,k+6}^{k+1,k+2,k+3,k+4,k+6} M(k+6) \nonumber \\
		& -12\del_{k+3,k+4,k+5,k+6}^{k+2,k+3,k+4,k+5} M(k+6) 
		+4 \del_{k+3,k+4,k+5,k+6}^{k+2,k+3,k+4,k+6} M(k+6)-\del_{k+4,k+5,k+6}^{k+2,k+3,k+6} M(k+6) \nonumber \\
		& + 4 \del_{k+4,k+5,k+6}^{k+2,k+4,k+5} M(k+6) 
		-2\del_{k+4,k+5,k+6}^{k+2,k+4,k+6}M(k+6)-4\del_{k+4,k+5,k+6}^{k+3,k+4,k+5} M(k+6) \nonumber \\
		& +\del_{k+4,k+5,k+6}^{k+3,k+4,k+6} M(k+6)
		-4 \del_{k+4,k+5,k+6}^{k+3,k+5,k+6} M(k+6) -2\del_{k+4,k+5,k+6}^{k+4,k+5,k+6} M(k+6).
	\end{align}

	Next, we expand the terms with three deletions along their bottom rows.
	\begin{align*}
		\del_{k+4,k+5,k+6}^{k+2,k+3,k+6} M(k+6) &= -\del_{k+3,k+4,k+5,k+6}^{k+1,k+2,k+3,k+6} M(k+6)+6\del_{k+3,k+4,k+5,k+6}^{k+2,k+3,k+5,k+6} M(k+6) \\
		\del_{k+4,k+5,k+6}^{k+2,k+4,k+5} M(k+6) &=-\del_{k+3,k+4,k+5,k+6}^{k+1,k+2,k+4,k+5} M(k+6)+\del_{k+3,k+4,k+5,k+6}^{k+2,k+3,k+4,k+5} M(k+6) \\
		& +2\del_{k+3,k+4,k+5,k+6}^{k+2,k+4,k+5,k+6}M(k+6) \\
		\del_{k+4,k+5,k+6}^{k+2,k+4,k+6}M(k+6)&=-\del_{k+3,k+4,k+5,k+6}^{k+1,k+2,k+4,k+6} M(k+6)+\del_{k+3,k+4,k+5,k+6}^{k+2,k+3,k+4,k+6} M(k+6) \\
		\del_{k+4,k+5,k+6}^{k+3,k+4,k+5} M(k+6) &= -\del_{k+3,k+4,k+5,k+6}^{k+1,k+3,k+4,k+5} M(k+6)-2\del_{k+3,k+4,k+5,k+6}^{k+2,k+3,k+4,k+5} M(k+6)\\
		& +2 \del_{k+3,k+4,k+5,k+6}^{k+3,k+4,k+5,k+6} M(k+6) \\
		\del_{k+4,k+5,k+6}^{k+3,k+4,k+6} M(k+6) &=-\del_{k+3,k+4,k+5,k+6}^{k+1,k+3,k+4,k+6} M(k+6)-2\del_{k+3,k+4,k+5,k+6}^{k+2,k+3,k+4,k+6} M(k+6)\\
		&+6 \del_{k+3,k+4,k+5,k+6}^{k+3,k+4,k+5,k+6} M(k+6) \\
		\del_{k+4,k+5,k+6}^{k+4,k+5,k+6} M(k+6) &=2 \del_{k+3,k+4,k+5,k+6}^{k+3,k+4,k+5,k+6} M(k+6)
	\end{align*}
	
	Also, we have $\del_{k+4,k+5,k+6}^{k+3,k+5,k+6} M(k+6) =0$ because this minor contains a column of zeroes.
	
	We substitute these identities into the previous expression for $ (-1)^{k} \Delta_5$ and simplify to obtain
	\begin{align}
		(-1)^{k} \Delta_5 &=  -4 \del_{k+2,k+3,k+4,k+5,k+6}^{k+1,k+2,k+3,k+4,k+5} M(k+6))+\del_{k+2,k+3,k+4,k+5,k+6}^{k+1,k+2,k+3,k+4,k+6} M(k+6) \nonumber \\
		&-\del_{k+3,k+4,k+5,k+6}^{k+1,k+2,k+3,k+6} M(k+6) 
		-4\del_{k+3,k+4,k+5,k+6}^{k+1,k+2,k+4,k+5} M(k+6)+2\del_{k+3,k+4,k+5,k+6}^{k+1,k+2,k+4,k+6} M(k+6) \nonumber \\
		& +4 \del_{k+3,k+4,k+5,k+6}^{k+1,k+3,k+4,k+5} M(k+6) 
		-\del_{k+3,k+4,k+5,k+6}^{k+1,k+3,k+4,k+6} M(k+6) +6\del_{k+3,k+4,k+5,k+6}^{k+2,k+3,k+5,k+6} M(k+6) \nonumber \\
		&  +8\del_{k+3,k+4,k+5,k+6}^{k+2,k+4,k+5,k+6} M(k+6).
	\end{align}

	Next, we analyze the terms with four deletions.
	
	The following minors each have a column of zeroes, hence their determinants are zero.
	\begin{align*}
		\del_{k+3,k+4,k+5,k+6}^{k+1,k+2,k+3,k+6} M(k+6) & = 0 \\
		\del_{k+3,k+4,k+5,k+6}^{k+1,k+2,k+4,k+5} M(k+6) &= 0 \\
		\del_{k+3,k+4,k+5,k+6}^{k+1,k+2,k+4,k+6} M(k+6) & = 0 \\
		\del_{k+3,k+4,k+5,k+6}^{k+2,k+3,k+5,k+6} M(k+6) &= 0 \\
		\del_{k+3,k+4,k+5,k+6}^{k+2,k+4,k+5,k+6} M(k+6) &= 0.
	\end{align*}
	
	We expand the following terms along their bottom rows.
	\begin{align*}
		\del_{k+3,k+4,k+5,k+6}^{k+1,k+3,k+4,k+5},M(k+6)) &=-\del_{k+2,k+3,k+4,k+5,k+6}^{k,k+1,k+3,k+4,k+5} M(k+6)+\del_{k+2,k+3,k+4,k+5,k+6}^{k+1,k+2,k+3,k+4,k+5} M(k+6)\nonumber \\
		&+2 \del_{k+2,k+3,k+4,k+5,k+6}^{k+1,k+3,k+4,k+5,k+6} M(k+6) \nonumber \\
		\del_{k+3,k+4,k+5,k+6}^{k+1,k+3,k+4,k+6} M(k+6) &= -\del_{k+2,k+3,k+4,k+5,k+6}^{k,k+1,k+3,k+4,k+6} M(k+6)+\del_{k+2,k+3,k+4,k+5,k+6}^{k+1,k+2,k+3,k+4,k+6} M(k+6)\nonumber \\
		&+8 \del_{k+2,k+3,k+4,k+5,k+6}^{k+1,k+3,k+4,k+5,k+6} M(k+6)
	\end{align*}
	
	We substitute these identities into the previous expression for $\Delta_5$ and simplify to obtain
	\begin{align}
		(-1)^{k} \Delta_5 = -4\del_{k+2,k+3,k+4,k+5,k+6}^{k,k+1,k+3,k+4,k+5} M(k+6)+\del_{k+2,k+3,k+4,k+5,k+6}^{k,k+1,k+3,k+4,k+6} M(k+6).
	\end{align}

	Each of these minors on the right hand side of this equation has a
	column of zeroes, hence their determinants are zero. Thus, $\Delta_5 =
	0$.
\end{proof}

\begin{lemma}
	The coefficient of the degree four term is
	\begin{equation} \label{eqn: h4}
		h_4 =  \frac{1}{3} (-1)^{k} \del_{k-1,k}^{k-1,k} M(k)
	\end{equation}
	for any $k\geq N_0$.
\end{lemma}

\begin{proof}
	When $h$ is a polynomial of degree at most four, we have
	\[
	h(l)-4 h(l+1)+6 h(l+2)-4 h(l+3)+h(l+4) = 24 h_4
	\]
	where $h_4$ is the coefficient of the degree four term.
	
	It is convenient to take $l = k+1$, as this allows us to use many of the identities
	established in the proof of the previous lemma. 
	\[
	h(k+1)-4 h(k+2)+6 h(k+3)-4 h(k+4)+h(k+5) = 24 h_4
	\]
	
	Substituting $h(N) = (-1)^{N+1} M(N+1)$ yields
	\begin{equation} \label{eqn: diff eq for lc Q}
		\det M(k+2)+4 \det M(k+3)+6 \det M(k+4)+4 \det M(k+5)+\det M(k+6) = (-1)^{k} 24 h_4
	\end{equation}
	
	We apply the identities established in the proof of the previous lemma
	to the left hand side of (\ref{eqn: diff eq for lc Q}) and obtain
	\[
	\det M(k+2)+4 \det M(k+3)+6 \det M(k+4)+4 \det M(k+5)+\det M(k+6) = 8 \del_{k+3,k+4,k+5,k+6}^{k+3,k+4,k+5,k+6} M(k+6).
	\]
	
	Expanding along the last four columns shows that $\del_{k+3,k+4,k+5,k+6}^{k+3,k+4,k+5,k+6} M(k+6) = \del_{k-1,k}^{k-1,k} M(k)$.
	
	We have
	\[
	(-1)^{k} 24 h_4 = 8 \del_{k-1,k}^{k-1,k} M(k),
	\]
	which yields the desired result.
\end{proof}

\begin{lemma}
	$\deg Q = 4$.
\end{lemma}
\begin{proof}
	By the previous lemma, the coefficient of the degree four term in $Q$
	is $\frac{1}{3} (-1)^{k} \del_{k-1,k}^{k-1,k} A(k-1)$. The matrix $\Del_{k-1,k}^{k-1,k} A(k-1)$ is the KKT
	matrix for finding the closest point on the cone $W \cap \{x_N =
	x_{N+1} = 0\}$ to the vector $a$ outside it. This problem has a unique solution,
	so this determinant cannot be zero.
\end{proof}

\subsection{$\chi$ is a polynomial}
\begin{definition}
	
	For each $n \geq 4$, define
	
	\[
	\Chi_n: \operatorname{Mat}_{n\times n} \rightarrow \operatorname{Mat}_{(n+1)\times (n+1)}
	\]
	as follows.
	\begin{itemize}
		\item Columns 1 through $(n-3)$ in $\Chi(M)$ are the same as in $M$, extended by 0 at the bottom.
		\item The last four entries in column $n-2$ in $\Chi(M)$ are $-1,2,-1,0$, and this column is 0 otherwise.    
		\item Column $n-1$ in $\Chi(M)$ is column $n-2$ in $M$, extended by  0 at the bottom.
		\item Column $n$ in $\Chi(M)$ is the sum of columns $n-1$ and $n$ in $M$, extended by 0 at the bottom.
		\item Column $n+1$ in $\Chi(M)$ is column $n$ from $M$, extended by 2 at the bottom.
	\end{itemize}
	
	We write $\Chi$ for the collection of maps $\{ \Chi_n\}$. We refer to
	$\Chi$ as a recurrence, and frequently omit the subscript.
	
\end{definition}

The application to our problem is as follows.
\begin{lemma}
	Suppose    $N \geq \max \{ \condIndex(\Gamma)+2,\max(I) +2 \}$. Then the matrices $A'(N,N-1)$ satisfy the recurrence $\Chi$. That is,
	$A'(N+1,N) = \Chi(A'(N,N-1))$.
\end{lemma}

\begin{lemma} \label{lem: chi is a polynomial of degree 3}
	Let $\{M(N)\}$ be a sequence of $N \times N$ matrices for $N \geq N_0$
	such that $M(N+1) = \Chi(M(N))$ for all $N \geq N_0$. Then $
	\chi(N) = (-1)^{N+1} \det M(N+1) $ is given by a polynomial of degree at most 3 for all $N \geq N_0+1$.
\end{lemma}

\begin{proof}
	We show that
	the following fourth-order difference equation vanishes for any integer $k \geq
	N_0$.
	\begin{equation} \label{eqn: Delta v1 chi}
		\chi(k) - 4\chi(k+1)+6\chi(k+2) -4\chi(k+3)+\chi(k+4) = 0.
	\end{equation}
	
	We define
	\[
	\Delta_4 := \chi(k) - 4\chi(k+1)+6\chi(k+2) -4\chi(k+3)+\chi(k+4) 
	\]
	Thus, our goal is to prove that $\Delta_4 = 0$.

	Substituting the definition of the function $\chi(N) = (-1)^{N+1} \det M(N+1) $
	yields the following expression.
	\begin{equation} \label{eqn: Delta v2 chi}
		(-1)^{k+1} \Delta_4 = \det M(k+1) + 4\det M(k+2)+6\det M(k+3) +4\det M(k+4)+\det
		M(k+5).
	\end{equation}
	
	Next, we find identities that will allow us to write each $\det M(k+i)$ in
	terms of the determinants of $M(k+5)$ and its minors.

	To begin, expanding $\det M(k+5)$ along its bottom row yields
	\begin{equation} \label{eqn: Mkp5 chi}
		\det M(k+5) = 2 \del_{k+5}^{k+5} M(k+5)
	\end{equation}
	
	Next, we relate $\det M(k+4)$ and determinants of minors of $M(k+5)$. Expanding $\det \aug(1,
	M(k+4))$ along its bottom row yields 
	\[
	\det \aug(1,M(k+4)) = -2\del_{k+4}^{k+3} \aug(1,M(k+4))  + 2 \del_{k+4}^{k+4} \aug(1,M(k+4)).
	\]
	But 
	\begin{align*}
		\Del_{k+4}^{k+3} \aug(1,M(k+4)) &= \Del_{k+4,k+5}^{k+2,k+4} M(k+5)
		\\
		\Del_{k+4}^{k+4} \aug(1,M(k+4)) &= \Del_{k+4,k+5}^{k+2,k+5} M(k+5) 
	\end{align*}
	so we obtain 
	\begin{equation} \label{eqn: Mkp4 chi}
		\det M(k+4) = -2 \del_{k+4,k+5}^{k+2,k+4} M(k+5)  + 2 \del_{k+4,k+5}^{k+2,k+5} M(k+5)
	\end{equation}

	In a similar fashion, we obtain the following identities.
	\begin{align}
		\det M(k+3) &= -4 \del_{k+3,k+4,k+5}^{k+1,k+2,k+4} M(k+5) + 2 \del_{k+3,k+4,k+5}^{k+1,k+2,k+5} M(k+5) \label{eqn: Mkp3 chi}\\
		\det M(k+2) &= -6 \del_{k+2,k+3,k+4,k+5}^{k,k+1,k+2,k+4} M(k+5)) + 2 \del_{k+2,k+3,k+4,k+5}^{k,k+1,k+2,k+5} M(k+5) \label{eqn: Mkp2 chi}\\
		\det M(k+1) &= -8 \del_{k+1,k+2,k+3,k+4,k+5}^{k-1,k,k+1,k+2,k+4} M(k+5)  + 2 \del_{k+1,k+2,k+3,k+4,k+5}^{k-1,k,k+1,k+2,k+5} M(k+5) \label{eqn: Mkp1 chi}
	\end{align}

	We substitute the identity (\ref{eqn: Mkp5 chi}) for $\det M(k+5)$ into the expression (\ref{eqn: Delta v2 chi}) for $(-1)^{k+1} \Delta_4$ to obtain the following. 
	\begin{equation} \label{eqn: Delta v3 chi}
		(-1)^{k+1} \Delta_4 =  \det M(k+1) + 4\det M(k+2)+6\det M(k+3) +4\det M(k+4)+2 \del_{k+5}^{k+5} M(k+5).
	\end{equation}
	
	We expand the term with one deletion. Expanding $\Del_{k+5}^{k+5} M(k+5)$ along tis bottom row yields
	\[
	\del_{k+5}^{k+5} M(k+5) = - \del_{k+4,k+5}^{k+2,k+5} M(k+5)+2\del_{k+4,k+5}^{k+4,k+5} M(k+5).
	\]
	We substitute this into the expression (\ref{eqn: Delta v3 chi}) for $(-1)^{k+1} \Delta_4$ and simplify.
	\begin{align} \label{eqn: Delta v4 chi}
		(-1)^{k+1} \Delta_4 &=  \det M(k+1) + 4\det M(k+2)+6\det M(k+3) +4\det M(k+4) \nonumber \\
		& - 2\del_{k+4,k+5}^{k+2,k+5} M(k+5)+4\del_{k+4,k+5}^{k+4,k+5} M(k+5).
	\end{align}
	
	Next, we substitute the identity (\ref{eqn: Mkp4 chi}) for $\det M(k+4)$ and simplify.
	\begin{align} \label{eqn: Delta v5 chi}
		(-1)^{k+1} \Delta_4 &=  \det M(k+1) + 4\det M(k+2)+6\det M(k+3) \nonumber \\
		&- 8\del_{k+4,k+5}^{k+2,k+4}+6\del_{k+4,k+5}^{k+2,k+5} M(k+5)+4\del_{k+4,k+5}^{k+4,k+5} M(k+5).
	\end{align}
	
	Next, we expand the terms with two deletions along their bottom rows.
	\begin{align*}
		\del_{k+4,k+5}^{k+2,k+4} M(k+5) &= -\del_{k+3,k+4,k+5}^{k+1,k+2,k+4} M(k+5)+2 \del_{k+3,k+4,k+5}^{k+2,k+4,k+5} M(k+5)  \\
		\del_{k+4,k+5}^{k+2,k+5} M(k+5) &= -\del_{k+3,k+4,k+5}^{k+1,k+2,k+5} M(k+5) +4 \del_{k+3,k+4,k+5}^{k+2,k+4,k+5} M(k+5) \\
		\del_{k+4,k+5}^{k+4,k+5} M(k+5) &=-\del_{k+3,k+4,k+5}^{k+1,k+4,k+5} M(k+5)-2 \del_{k+3,k+4,k+5}^{k+2,k+4,k+5} M(k+5)
	\end{align*}
	
	We substitute these identities into the expression (\ref{eqn: Delta v5 chi}) for $(-1)^{k+1} \Delta_4$ and simplify.
	\begin{align} \label{eqn: Delta v6 chi}
		(-1)^{k+1} \Delta_4 &=  \det M(k+1) + 4\det M(k+2)+6\det M(k+3) \nonumber \\
		&- 8 \del_{k+3,k+4,k+5}^{k+1,k+2,k+4} M(k+5)) -6\del_{k+3,k+4,k+5}^{k+1,k+2,k+5} M(k+5)-4 \del_{k+3,k+4,k+5}^{k+1,k+4,k+5} M(k+5).
	\end{align}
	
	We substitute the identity (\ref{eqn: Mkp3 chi}) for $\det M(k+3)$ and simplify: 
	\begin{align} \label{eqn: Delta v7 chi}
		(-1)^{k+1} \Delta_4 &=  \det M(k+1) + 4\det M(k+2)- 16 \del_{k+3,k+4,k+5}^{k+1,k+2,k+4} M(k+5)) \nonumber \\
		& +6\del_{k+3,k+4,k+5}^{k+1,k+2,k+5} M(k+5)-4 \del_{k+3,k+4,k+5}^{k+1,k+4,k+5} M(k+5).
	\end{align}
	
	Next, we expand the terms with three deletions along their bottom rows.
	\begin{align*}
		\del_{k+3,k+4,k+5}^{k+1,k+2,k+4} M(k+5)) &= -\del_{k+2,k+3,k+4,k+5}^{k,k+1,k+2,k+4} M(k+5)+2 \del_{k+2,k+3,k+4,k+5}^{k+1,k+2,k+4,k+5} M(k+5)
		\\
		\del_{k+3,k+4,k+5}^{k+1,k+2,k+5} M(k+5) &= -\del_{k+2,k+3,k+4,k+5}^{k,k+1,k+2,k+5} M(k+5)+6 \del_{k+2,k+3,k+4,k+5}^{k+1,k+2,k+4,k+5} M(k+5) 
		\\
		\del_{k+3,k+4,k+5}^{k+1,k+4,k+5} M(k+5)&=-\del_{k+2,k+3,k+4,k+5}^{k,k+1,k+4,k+5} M(k+5)+\del_{k+2,k+3,k+4,k+5}^{k+1,k+2,k+4,k+5} M(k+5).
	\end{align*}
	
	We substitute these identities into the expression (\ref{eqn: Delta v7 chi}) for $(-1)^{k+1} \Delta_4$ and simplify.
	\begin{align} \label{eqn: Delta v8 chi}
		(-1)^{k+1} \Delta_4 &=  \det M(k+1) + 4\det M(k+2)+16\del_{k+2,k+3,k+4,k+5}^{k,k+1,k+2,k+4} M(k+5) \nonumber \\
		& -6 \del_{k+2,k+3,k+4,k+5}^{k,k+1,k+2,k+5} M(k+5) +4\del_{k+2,k+3,k+4,k+5}^{k,k+1,k+4,k+5} M(k+5).
	\end{align}
	
	We substitute the identity (\ref{eqn: Mkp2 chi}) for $\det M(k+2)$ and simplify.
	\begin{align} \label{eqn: Delta v9 chi}
		(-1)^{k+1} \Delta_4 &=  \det M(k+1) -8\del_{k+2,k+3,k+4,k+5}^{k,k+1,k+2,k+4} M(k+5) \nonumber \\
		& +2 \del_{k+2,k+3,k+4,k+5}^{k,k+1,k+2,k+5} M(k+5) +4\del_{k+2,k+3,k+4,k+5}^{k,k+1,k+4,k+5} M(k+5).
	\end{align}
	
	We expand the terms with four deletions along their bottom rows.
	\begin{align*}
		\del_{k+2,k+3,k+4,k+5}^{k,k+1,k+2,k+4} M(k+5) &= (-1) \del_{k+1,k+2,k+3,k+4,k+5}^{k-1,k,k+1,k+2,k+4} M(k+5) +2 \del_{k+1,k+2,k+3,k+4,k+5}^{k,k+1,k+2,k+4,k+5} M(k+5) \\
		\del_{k+2,k+3,k+4,k+5}^{k,k+1,k+2,k+5} M(k+5) &= -\del_{k+1,k+2,k+3,k+4,k+5}^{k-1,k,k+1,k+2,k+5} M(k+5)+8 \del_{k+1,k+2,k+3,k+4,k+5}^{k,k+1,k+2,k+4,k+5} M(k+5)
	\end{align*}
	Note that $\Del_{k+2,k+3,k+4,k+5}^{k,k+1,k+4,k+5} M(k+5)$ has a column of zeroes, hence this minor has determinant 0.
	
	We substitute these identities into the expression (\ref{eqn: Delta v9 chi}) for $(-1)^{k+1} \Delta_4$ and simplify.
	\begin{align} \label{eqn: Delta v10 chi}
		(-1)^{k+1} \Delta_4 &=   \det M(k+1) +8\del_{k+1,k+2,k+3,k+4,k+5}^{k-1,k,k+1,k+2,k+4} M(k+5) \nonumber \\
		& -2 \del_{k+1,k+2,k+3,k+4,k+5}^{k-1,k,k+1,k+2,k+5} M(k+5).
	\end{align}
	
	Substituting the identity (\ref{eqn: Mkp4 chi}) for $\det M(k+1)$ yields $\Delta_4 = 0$. \end{proof}

\begin{lemma} \label{lem: chi3}
	The coefficient of the degree three term is
	\begin{equation} \label{eqn: chi3}
		\chi_3 =  \frac{2}{3}(-1)^{k+1} \left(\del_{k,k+1}^{k,k+1} M(k+1) + \del_{k-1,k+1}^{k,k+1} M(k+1) \right)
	\end{equation}
	for any $k\geq N_0$.
\end{lemma}

\begin{proof}
	When $\chi$ is a polynomial of degree at most three, we have
	\begin{equation} 
		\chi(l)-3 \chi(l+1)+3\chi(l+2)-\chi(l+3) = -6 \chi_3
	\end{equation}
	where $\chi_3$ is the coefficient of the degree three term.
	
	It is convenient to take $l = k+1$, as this allows us to use many of the identities
	established in the proof of the previous lemma. 
	\begin{equation} 
		\chi(k+1)-3 \chi(k+2)+3\chi(k+3)-\chi(k+4) = -6 \chi_3
	\end{equation}

	Substituting $\chi(N) = (-1)^{N+1} M(N+1)$ yields
	\begin{equation} \label{eqn: diff eq for lc chi}
		\det M(k+2)+3 \det M(k+3)+3 \det M(k+4)+\det M(k+5)= (-1)^{k+1} 6 \chi_3
	\end{equation}

	We apply the identities established in the proof of the previous lemma
	to the left hand side of (\ref{eqn: diff eq for lc chi}) and obtain
	\begin{align} \label{eqn: output from diff eq chi lc}
		\lefteqn{\det M(k+2)+3 \det M(k+3)+3 \det M(k+4)+\det M(k+5)} \nonumber \\
		& = -4 \del_{k+1,k+2,k+3,k+4,k+5}^{k-1,k+1,k+2,k+4,k+5} M(k+5)-4\del_{k+1,k+2,k+3,k+4,k+5}^{k,k+1,k+2,k+4,k+5} M(k+5) 
	\end{align}
	
	We examine each of the terms on the right hand side of (\ref{eqn:
		output from diff eq chi lc}).

	Expanding $\Del_{k+1,k+2,k+3,k+4,k+5}^{k-1,k+1,k+2,k+4,k+5} M(k+5)$ along column $k$ yields
	\[
	\del_{k+1,k+2,k+3,k+4,k+5}^{k-1,k+1,k+2,k+4,k+5} M(k+5) =\del_{k,k+1,k+2,k+3,k+4,k+5}^{k-1,k,k+1,k+2,k+4,k+5} M(k+5)
	\]
	and
	\[
	\Del_{k,k+1,k+2,k+3,k+4,k+5}^{k-1,k,k+1,k+2,k+4,k+5} M(k+5)=\Del_{k,k+1}^{k,k+1} M(k+1).
	\]
	Thus we have
	\begin{equation} \label{eqn: Mkp5 to Mkp1 1}
		\del_{k+1,k+2,k+3,k+4,k+5}^{k-1,k+1,k+2,k+4,k+5} M(k+5) = del({k,k+1},{k,k+1},M(k+1)).
	\end{equation}
	
	Expanding $\Del_{k+1,k+2,k+3,k+4,k+5}^{k,k+1,k+2,k+4,k+5} M(k+5)$ along column $k-1$ yields
	\[
	\del_{k+1,k+2,k+3,k+4,k+5}^{k,k+1,k+2,k+4,k+5} M(k+5) =-\del_{k-1,k+1,k+2,k+3,k+4,k+5}^{k-1,k,k+1,k+2,k+4,k+5} M(k+5)-2\del_{k,k+1,k+2,k+3,k+4,k+5}^{k-1,k,k+1,k+2,k+4,k+5} M(k+5)
	\]
	But 
	\begin{align*}
		\Del_{k-1,k+1,k+2,k+3,k+4,k+5}^{k-1,k,k+1,k+2,k+4,k+5} M(k+5) &= \Del_{k-1,k+1}^{k,k+1} M(k+1) \\
		\Del_{k,k+1,k+2,k+3,k+4,k+5}^{k-1,k,k+1,k+2,k+4,k+5} M(k+5) &= \Del_{k,k+1}^{k,k+1} M(k+1)
	\end{align*}
	Thus we have
	\begin{equation} \label{eqn: Mkp5 to Mkp1 2}
		\del_{k+1,k+2,k+3,k+4,k+5}^{k,k+1,k+2,k+4,k+5} M(k+5)=-\del_{k-1,k+1}^{k,k+1} M(k+1)-2 \del_{k,k+1}^{k,k+1} M(k+1).
	\end{equation}

	We combine equations (\ref{eqn: diff eq for lc chi}), (\ref{eqn: output from diff eq chi lc}), (\ref{eqn: Mkp5 to Mkp1 1}), and (\ref{eqn: Mkp5 to Mkp1 2}).
	\begin{align*}
		(-1)^{k+1} 6 \chi_3 &= \det M(k+2)+3 \det M(k+3)+3 \det M(k+4)+\det M(k+5)\\
		&= -4 \del_{k+1,k+2,k+3,k+4,k+5}^{k-1,k+1,k+2,k+4,k+5} M(k+5) -4\del_{k+1,k+2,k+3,k+4,k+5}^{k,k+1,k+2,k+4,k+5} M(k+5) \\
		&= 4\del_{k-1,k+1}^{k,k+1} M(k+1)+4 \del_{k,k+1}^{k,k+1} M(k+1).  
	\end{align*}
	
	This gives the desired result:
	\[
	\chi_3 = \frac{2}{3} (-1)^{k+1} \left( \del_{k-1,k+1}^{k,k+1} M(k+1)+ \del_{k,k+1}^{k,k+1} M(k+1)\right).  
	\]
	
\end{proof}

Fix a positive integer $k$. We use the following notation for the Taylor expansions of this polynomial centered at $k$.
\begin{eqnarray*}
	\chi(N) &=& \chi_3 (N-k)^3 + \chi_2 (N-k)^2 +  \chi_1 (N-k) + \chi_0 
\end{eqnarray*}

\begin{lemma} \label{lem: chi2}
	The coefficient of the degree two term in this Taylor expansion is
	\[
	\chi_2  = (-1)^{k+1} \left(2 \del_{k-1,k+1}^{k,k+1} M(k+1) +4 \del_{k,k+1}^{k,k+1} M(k+1) \right).
	\] 
\end{lemma}

\begin{proof}
	Evaluating the Taylor expansion for $N=k$, $N=k+1$, and $N=k+2$ yields
	\begin{align*}
		\chi(k) &= \chi_0 \\
		\chi(k+1) &= \chi_3 + \chi_2 + \chi_1 + \chi_0 \\
		\chi(k+2) &= 8 \chi_3 + 4 \chi_2 + 2 \chi_1 + \chi_0.
	\end{align*}
	
	Solving this system for $\chi_2$ yields
	\[
	2 \chi_2 = \chi(k+2) - 2 \chi(k+1) + \chi(k) - 6 \chi_3.
	\]
	Substituting $\chi(N) = (-1)^{N+1} M(N+1)$ yields
	\begin{equation} \label{eqn: diff eq for chi2}
		2 \chi_2 = (-1)^{k+1}( \det M(k+3) + 2 \det M(k+2) + M(k+1))  - 6 \chi_3.
	\end{equation}
	
	We use the identities established in the two previous proofs as well as one
	additional identity.
	\[
	\del_{k+2,k+3,k+4,k+5}^{k+1,k+2,k+4,k+5} M(k+5)
	=-\del_{k+1,k+2,k+3,k+4,k+5}^{k-1,k+1,k+2,k+4,k+5} M(k+5) -2
	\del_{k+1,k+2,k+3,k+4,k+5}^{k,k+1,k+2,k+4,k+5} M(k+5).
	\]
	
	This allows us to write 
	\begin{equation} \label{eqn: output from diff eq for chi2}
		\det M(k+3) + 2 \det M(k+2) + M(k+1) = 8\del_{k-1,k+1}^{k,k+1} M(k+1)+12 \del_{k,k+1}^{k,k+1} M(k+1).  
	\end{equation}
	
	Substituting (\ref{eqn: output from diff eq for chi2}) and the expression (\ref{eqn: chi3}) previously found for $\chi_3$ into (\ref{eqn: diff eq for chi2}) yields
	\begin{align*}
		2 \chi_2 &= (-1)^{k+1}( \det M(k+3) + 2 \det M(k+2) + M(k+1))  - 6 \chi_3 \\
		&=  (-1)^{k+1}( 8\del_{k-1,k+1}^{k,k+1} M(k+1)+12
		\del_{k,k+1}^{k,k+1} M(k+1)) \\
		& \mbox{} - 6 \frac{2}{3} (-1)^{k+1} \left( \del_{k-1,k+1}^{k,k+1} M(k+1)+ \del_{k,k+1}^{k,k+1} M(k+1)\right) \\  
		&= (-1)^{k+1} (4\del_{k-1,k+1}^{k,k+1} M(k+1) +8 \del_{k,k+1}^{k,k+1} M(k+1) ).
	\end{align*}

\end{proof}

\begin{lemma} \label{lem: chi1}
	The coefficient of the degree one term in this Taylor expansion is
	\[
	\chi_1 = \frac{1}{3}(-1)^{k+2} \left( 8 \del_{k-1,k+1}^{k,k+1}
	M(k+1)  + 14 \del_{k,k+1}^{k,k+1} M(k+1) + 6 \del_{k+2}^{k+1}M(k+2)  \right).
	\] 
\end{lemma}

\begin{proof}
	Evaluating the Taylor expansion for $N=k+1$ yields
	\[
	\chi(k+1) = \chi_3 + \chi_2 + \chi_1 + \chi_0 
	\]
	Hence,
	\begin{align}
		\chi_1 &= \chi(k+1) - \chi_0 - \chi_3 -\chi_2 \nonumber \\
		&= (-1)^{k+2} \det M(k+2) - (-1)^{k+1} \det M(k+1) - \chi_3 - \chi_2
		\nonumber \\
		&= (-1)^{k+2} \left( \det M(k+2) + \det M(k+1) \right) - \chi_3 -
		\chi_2. \label{eqn: chi1 formula}
	\end{align}

	We prove the following identity:
	\begin{equation} \label{eqn: Mkp2 plus Mkp1}
		\det M(k+2) + \det M(k+1) =  \del_{k+2}^{k+1}M(k+2).
	\end{equation}
	
	We obtain this as follows. Start by expanding $\red M(k+2)$ along its
	bottom row.
	\begin{equation} \label{eqn: det red Mkp2 expansion}
		\det \red(M(k+2)) =2 \del_{k+2}^{k+1} \red M(k+2)+2 \del_{k+2}^{k+2}
		\red M(k+2)
	\end{equation} 
	We have
	\begin{align}
		\det \red M(k+2) &= \det M(k+2) \label{eqn: det red Mkp2 term 1}\\
		\del_{k+2}^{k+1} \red M(k+2) &= \del_{k+2}^{k+1} M(k+2) \label{eqn: det red Mkp2 term 2}
	\end{align}
	
	Expanding $\Del_{k+2}^{k+2} \red M(k+2)$ along its bottom row yields
	\[
	\del_{k+2}^{k+2} \red M(k+2) =-\del_{k+1,k+2}^{k-1,k+2} \red M(k+2).
	\]
	Expanding $\det M(k+1)$ along its bottom row yields
	\[
	\det M(k+1) = 2 \del_{k+1}^{k+1} M(k+1),
	\]
	and we have 
	\[
	\Del_{k+1,k+2}^{k-1,k+2} \red M(k+2) = \Del_{k+1}^{k+1} M(k+1).
	\]
	Thus
	\begin{equation} \label{eqn: det red Mkp2 term 3}
		2 \del_{k+2}^{k+2} \red M(k+2) = - \det M(k+1).
	\end{equation}
	Substituting (\ref{eqn: det red Mkp2 term 1}), (\ref{eqn: det red Mkp2
		term 2}), and (\ref{eqn: det red Mkp2 term 3}) into (\ref{eqn: det
		red Mkp2 expansion}) yields (\ref{eqn: Mkp2 plus Mkp1}).

	Substituting (\ref{eqn: Mkp2 plus Mkp1}) and the formulas for $\chi_3$
	and $\chi_2$ from the previous lemmas yields the result. 
	
\end{proof}

We summarize the results of this section with the following lemma. 
\begin{lemma} \label{lem: formulas for chi}
	Suppose that $k \geq \max \{ \condIndex(\Gamma)+2,\max(I) +2 \}$. Then
	\begin{align*}
		\chi_3 & = \frac{1}{3}(-1)^{k+1} \left(2 \del_{k,k+1}^{k,k+1} A'(k,k-1) + 2 \del_{k-1,k+1}^{k,k+1} A'(k,k-1) \right) \\
		\chi_2 & = (-1)^{k+1} \left(4 \del_{k,k+1}^{k,k+1} A'(k,k-1) + 2 \del_{k-1,k+1}^{k,k+1} A'(k,k-1) \right) \\
		\chi_1 &= \frac{1}{3}(-1)^{k+1} \left( -14 \del_{k,k+1}^{k,k+1} A'(k,k-1) -8 \del_{k-1,k+1}^{k,k+1} A'(k,k-1) - 6 \del_{k+2}^{k+1} A'(k+1,k) \right)\\
		\chi_0 &= (-1)^{k+1} \det A'(k,k-1).
	\end{align*}
	
\end{lemma}
\begin{proof}
	We apply Lemmas \ref{lem: chi3}, \ref{lem: chi2}, and \ref{lem: chi1}
	with $M(k+1) = A'(k,k-1)$.
\end{proof}

\begin{lemma}
	$\deg \omega = 4$.
\end{lemma}
\begin{proof}
	The last row of the KKT matrix equation says $-x_{N-1} +2 x_{N+1} =
	4$. Clearing denominators yields $-\chi + 2 \omega = 4Q$. Since $\deg
	\omega = 4$ and $\deg \chi \leq 3$, the result follows. 
\end{proof}

\subsection{$\psi$ is a polynomial}

\begin{definition}
	For each $n \geq 3$, define
	\[
	\Psi_n : \operatorname{Mat}_{n\times n} \rightarrow \operatorname{Mat}_{(n+1)\times (n+1)}
	\]
	as follows.
	\begin{itemize}
		\item Columns 1 through $(n-2)$ in $\Psi(M)$ are the same as in $M$, extended by 0 at the bottom.
		\item The last three entries in column $n-1$ in $\Psi(M)$ are $-1,2,-1$, and this column is 0 otherwise.    
		\item Column $n$ in $\Psi(M)$ is column $n-1$ in $M$, extended by  0 at the bottom.
		\item Column $n+1$ in $\Psi(M)$ is column $n$ from $M$, extended by 2 at the bottom.
	\end{itemize}
	
	We write $\Psi$ for the collection of maps $\{ \Psi_n\}$.    We refer to
	$\Psi$ as a recurrence, and frequently omit the subscript.
\end{definition}

The application to our problem is as follows.
\begin{lemma}
	Suppose    $N \geq \max \{ \condIndex(\Gamma)+2,\max(I) +2 \}$. Then the matrices $A'(N,N)$ satisfy the recurrence $\Psi$. That is,
	$A'(N+1,N+1) = \Psi(A'(N,N))$.
\end{lemma}

\begin{lemma} \label{lem: psi is a polynomial of degree 2}
	Let $\{M(N)\}$ be a sequence of $N \times N$ matrices for $N \geq N_0$
	such that $M(N+1) = \Psi(M(N))$ for all $N \geq N_0$. Then $
	\psi(N) = (-1)^{N+1} \det M(N+1) $ is given by a polynomial of degree at
	most two for all $N \geq N_0$.
\end{lemma}

\begin{proof}
	We show that
	the following third-order difference equation vanishes for any integer $k \geq
	N_0+1$.
	\begin{equation} \label{eqn: Delta v1 psi}
		\psi(k) - 3\psi(k+1) +3\psi(k+2) -\psi(k+3) = 0.
	\end{equation}

	We define
	\[
	\Delta_3 := \psi(k) - 3\psi(k+1)+3\psi(k+2) -\psi(k+3).
	\]
	Thus, our goal is to prove that $\Delta_3 = 0$.

	Substituting the definition of the function $\psi(N) = (-1)^{N+1} \det M(N+1) $
	yields the following expression for $\Delta_3$.
	\begin{equation} \label{eqn: Delta v2 psi}
		(-1)^{k+1} \Delta_3 = \det M(k+1) + 3\det M(k+2)+3\det M(k+3) +\det M(k+4).
	\end{equation}
	
	Next, we find identities that will allow us to write each $\det M(k+i)$ in
	terms of the determinants of $M(k+4)$ and its minors.

	To begin, expanding $\det M(k+4)$ along its bottom row yields
	\[
	\det M(k+4) = -\del_{k+4}^{k+2} M(k+4) +2\del_{k+4}^{k+4} M(k+4).
	\]
	But
	\[
	\Del_{k+4}^{k+2} M(k+4) = M(k+3)
	\]
	so we obtain
	\begin{equation} \label{eqn: Mkp2 psi}
		\det M(k+3) = -\det M(k+4) + 2\del_{k+4}^{k+4} M(k+4).
	\end{equation}

	In a similar fashion, we obtain the following identities.
	\begin{align}
		\det M(k+2) &= -\det M(k+3) + 2\del_{k+3,k+4}^{k+2,k+4} M(k+4). \label{eqn: Mkp1 psi}\\
		\det M(k+1) &= -\det M(k+2)+ 2\del_{k+2,k+3,k+4}^{k+1,k+2,k+4} M(k+4)\label{eqn: Mk psi}
	\end{align}

	We substitute the identities (\ref{eqn: Mkp2 psi}), (\ref{eqn: Mkp1 psi}), and (\ref{eqn: Mkp2 psi}) into the expression (\ref{eqn: Delta v2 psi}) for $(-1)^{k+1}\Delta_3$ to obtain the following. 
	\begin{equation} \label{eqn: Delta v3 psi}
		(-1)^{k+1}\Delta_3 =  2 \del_{k+2,k+3,k+4}^{k+1,k+2,k+4} M(k+4) +4 \del_{k+3,k+4}^{k+2,k+4} M(k+4)+2 \del_{k+4}^{k+4} M(k+4).
	\end{equation}
	
	We expand the term with one deletion. Expanding $\Del_{k+4}^{k+4} M(k+4)$ along its bottom row yields
	\[
	\del_{k+4}^{k+4} M(k+4) = -\del_{k+3,k+4}^{k+1,k+4} M(k+4) -2 \del_{k+3,k+4}^{k+2,k+4} M(k+4).
	\]
	We substitute this into the expression (\ref{eqn: Delta v3 psi}) for $(-1)^{k+1}\Delta_3$ and simplify.
	\begin{equation} \label{eqn: Delta v4 psi}
		(-1)^{k+1}\Delta_3 =  \del_{k+2,k+3,k+4}^{k+1,k+2,k+4} M(k+4) -\del_{k+3,k+4}^{k+1,k+4} M(k+4).
	\end{equation}
	
	The minor $\Del_{k+3,k+4}^{k+1,k+4} M(k+4)$ has only one nonzero entry in column $k+2$. Expanding along this column yields
	\[
	\del_{k+3,k+4}^{k+1,k+4} M(k+4) = \del_{k+2,k+3,k+4}^{k+1,k+2,k+4} M(k+4).
	\]
	Thus, $(-1)^{k+1}\Delta_3 = 0$. This completes the proof that $\psi(N)$ is a polynomial of degree at most two.
\end{proof}

\begin{lemma}\label{lem: psi2}
	The coefficient of the degree two term is 
	\[
	\psi_2 = (-1)^{k+1} \left(\del_{k+1}^{k+1}  M(k+1) + \del_{k}^{k+1}
	M(k+1) \right)
	\]
	for any $k\geq N_0$.
\end{lemma}
\begin{proof}
	When $\psi$ is a polynomial of degree at most two, we have
	\begin{equation} 
		\psi(k+1)-2 \psi(k+2)+\psi(k+3)= 2 \psi_2
	\end{equation}
	where $\psi_2$ is the coefficient of the degree two term.
	
	Substituting $\psi(N) = (-1)^{N+1} M(N+1)$ yields
	\begin{equation} \label{eqn: diff eq for lc psi}
		\det M(k+1)+2 \det M(k+2)+ \det M(k+3) = (-1)^{k+1} 2 \psi_2
	\end{equation}

	We can argue as we did in the previous proof to obtain the following
	identities.
	\begin{align*}
		\det M(k+2) &= -\det M(k+3) +2\del_{k+3}^{k+3} M(k+3) \\
		\det M(k+1) &= -\det M(k+2) +2 \del_{k+2,k+3}^{k+1,k+3} M(k+3) \\
		\del_{k+3}^{k+3} M(k+3) &=-\del_{k+2,k+3}^{k,k+3} M(k+3)
		-2\del_{k+2,k+3}^{k+1,k+3} M(k+3).
	\end{align*}
	
	We use these identities to evaluate the left hand side of (\ref{eqn:
		diff eq for lc psi}) and obtain
	\begin{equation} \label{eqn: output from diff eq psi}
		\det M(k+1)+2 \det M(k+2)+ \det M(k+3) = -2
		\del_{k+2,k+3}^{k+1,k+3} M(k+3) -2 \del_{k+2,k+3}^{k,k+3} M(k+3)
	\end{equation}
	
	Expanding the first term on the right in (\ref{eqn: output from diff eq psi}) along its bottom row yields
	\begin{equation} \label{eqn: psi2 stuff}
		\del_{k+2,k+3}^{k+1,k+3} M(k+3) = -\del_{k+1,k+2,k+3}^{k-1,k+1,k+3}
		M(k+3)-2 \del_{k+1,k+2,k+3}^{k,k+1,k+3} M(k+3)
	\end{equation}
	Expanding the first term on the right in (\ref{eqn: psi2 stuff}) along its bottom row yields
	\[
	\del_{k+1,k+2,k+3}^{k-1,k+1,k+3} M(k+3) =
	\del_{k,k+1,k+2,k+3}^{k-1,k,k+1,k+3} M(k+3)
	\]
	and
	\[
	\Del_{k,k+1,k+2,k+3}^{k-1,k,k+1,k+3} M(k+3) =\del_{k,k+1}^{k-1,k+1} M(k+1).
	\]
	For the second term on the right in (\ref{eqn: psi2 stuff}), we have
	\[
	\Del_{k+1,k+2,k+3}^{k,k+1,k+3} M(k+3) = \Del_{k+1}^{k+1} M(k+1).
	\]
	Hence, we obtain 
	\begin{equation} \label{eqn: psi2 term 1}
		\del_{k+2,k+3}^{k+1,k+3} M(k+3) = -\del_{k,k+1}^{k-1,k+1} M(k+1) -2
		\del_{k+1}^{k+1} M(k+1)
	\end{equation}

	Expanding the second term on the right in (\ref{eqn: output from diff eq psi}) along its bottom row yields
	\begin{equation} \label{eqn: psi2 stuff 2}
		\del_{k+2,k+3}^{k,k+3} M(k+3) =-\del_{k+1,k+2,k+3}^{k-1,k,k+3} M(k+3)
		+\del_{k+1,k+2,k+3}^{k,k+1,k+3} M(k+3).
	\end{equation}
	The first term on the right in (\ref{eqn: psi2 stuff 2}) has a column
	of zeroes, hence its determinant is zero. For the second term on the
	right in  (\ref{eqn: psi2 stuff 2}), we have
	\[
	\Del_{k+1,k+2,k+3}^{k,k+1,k+3} M(k+3) = \Del_{k+1}^{k+1} M(k+1).
	\]
	Hence, we obtain 
	\begin{equation} \label{eqn: psi2 term 2}
		\del_{k+2,k+3}^{k,k+3} M(k+3) = \del_{k+1}^{k+1} M(k+1).
	\end{equation}
	
	Substituting (\ref{eqn: psi2 term 1}) and (\ref{eqn: psi2 term 2})
	into (\ref{eqn: output from diff eq psi})
	yields
	\[
	\det M(k+1)+2 \det M(k+2)+ \det M(k+3) = 2 \del_{k,k+1}^{k-1,k+1}
	M(k+1)+2 \del_{k+1}^{k+1} M(k+1).
	\]
	
	Finally, expanding along the bottom row yields
	\[
	\del_{k}^{k+1} M(k+1)) = \del_{k,k+1}^{k-1,k+1} M(k+1).
	\]

	Thus, we have
	\begin{align*}
		(-1)^{k+1} 2 \psi_2 &= \det M(k+1)+2 \det M(k+2)+ \det M(k+3)\\
		&= 2 \del_{k,k+1}^{k-1,k+1}
		M(k+1)+2 \del_{k+1}^{k+1} M(k+1).\\
		&= 2 \del_{k}^{k+1} M(k+1)) +2 \del_{k+1}^{k+1} M(k+1).
	\end{align*}
	
	This gives the desired result:
	\[
	\psi_2 = (-1)^{k+1} \left( \del_{k}^{k+1} M(k+1)) +  \del_{k+1}^{k+1} M(k+1)\right).  
	\]
\end{proof}

Fix a positive integer $k$. We use the following notation for the Taylor expansions of this polynomial centered at $k$.
\begin{eqnarray*}
	\psi(N) &=& \psi_2 (N-k)^2 +\psi_1 (N-k)+ \psi_0
\end{eqnarray*}

\begin{lemma} \label{lem: psi1}
	The coefficient of the degree 1 term in this Taylor expansion is
	\[
	\psi_1  = (-1)^{k+1} \left(3\del_{k+1}^{k+1}  M(k+1) + \del_{k}^{k+1}  M(k+1) \right).
	\]
\end{lemma}
\begin{proof}
	
	We have
	\[
	\psi = \psi_2 (N-k)^2 + \psi_1 (N-k) + \psi_0.
	\]
	When $N = k$, we have $\psi(k) = \psi_0$, and when $N = k+1$ we have
	\[
	\psi(k+1) = \psi_2 + \psi_1 + \psi_0,
	\]
	so
	\begin{align}
		\psi_1 &= \psi(k+1) - \psi_2 - \psi_0 \nonumber \\
		&= (-1)^{k+2} \det M(k+2) - \psi_2 - (-1)^{k+1} \det M(k+1) \nonumber\\
		&= (-1)^{k+2} (\det M(k+2) + \det M(k+1)) - \psi_2.
	\end{align}
	
	Expanding $\det M(k+2)$ along its bottom row yields
	\[
	\det M(k+2) = -\del_{k+2}^{k} M(k+2) +2\del_{k+2}^{k+2} M(k+2).
	\]
	But
	\[
	\Del_{k+2}^{k} M(k+2) = M(k+1)
	\]
	so we obtain
	\begin{equation} 
		\det M(k+2) + \det M(k+1) = 2\del_{k+2}^{k+2} M(k+2).
	\end{equation}
	Substituting this and the expression for $\psi_2$ yields
	\begin{equation} \label{eqn: psi1 M(k+2)}
		\psi_1 = (-1)^{k+2} 2\del_{k+2}^{k+2} M(k+2) - (-1)^{k+1} \left( \del_{k}^{k+1} M(k+1)) +  \del_{k+1}^{k+1} M(k+1)\right). 
	\end{equation}
	
	We study the first term on the right in (\ref{eqn: psi1 M(k+2)}). Expanding along column $k$ yields
	\[
	\del_{k+2}^{k+2} M(k+2) =-\del_{k,k+2}^{k,k+2} M(k+2)-2\del_{k+1,k+2}^{k,k+2} M(k+2).
	\]
	But
	\begin{align*}
		\Del_{k,k+2}^{k,k+2} M(k+2) &= \Del_{k}^{k+1} M(k+1) \\
		\Del_{k+1,k+2}^{k,k+2} M(k+2) &= \Del_{k+1}^{k+1} M(k+1)
	\end{align*}
	so
	\[
	\del_{k+2}^{k+2} M(k+2) = -\del_{k}^{k+1} M(k+1) -2\del_{k+1}^{k+1} M(k+1).
	\]
	Substituting this into (\ref{eqn: psi1 M(k+2)}) yields
	\begin{align*}
		\psi_1 &= (-1)^{k+2} (-2\del_{k}^{k+1} M(k+1) -4\del_{k+1}^{k+1} M(k+1)- (-1)^{k+1} \left( \del_{k}^{k+1} M(k+1)) +  \del_{k+1}^{k+1} M(k+1)\right). \\
		&= (-1)^{k+1} \left( 3\del_{k+1}^{k+1} M(k+1) + \del_{k}^{k+1} M(k+1) \right),
	\end{align*}
	as desired.
\end{proof}

We summarize the results of this section with the following lemma. 
\begin{lemma} \label{lem: formulas for psi}
	Suppose that $k \geq \max \{ \condIndex(\Gamma)+2,\max(I) +2 \}$. Then
	\  \begin{align*}
		\psi_2 & = (-1)^{k+1} \left(\del_{k+1}^{k+1}  A'(k,k) + \del_{k}^{k+1}  A'(k,k) \right) \\
		\psi_1 & = (-1)^{k+1} \left(3\del_{k+1}^{k+1}  A'(k,k) + \del_{k}^{k+1}  A'(k,k) \right) \\
		\psi_0 &= (-1)^{k+1} \det A'(k,k).
	\end{align*}
	
\end{lemma}
\begin{proof}
	We apply Lemmas \ref{lem: psi2} and \ref{lem: psi1}
	with $M(k+1) = A'(k,k)$.
\end{proof}

\section{Proof of Proposition \ref{prop: worst 1ps for
		cusps}}

Here we give a proof of Propositon \ref{prop: worst 1ps for
	cusps}, concerning worst 1-PS's for higher order cusps.

\subsection{Statement of the main result for cusps}
\begin{definition}
	For any integer $r \geq 1$, we define the polynomial
	\begin{equation}
		f(r,x):= (4r-2)x^3+(6r-6)x^2-(12r^3+6r^2+4r+4)x-(30r^3+18r^2).   
	\end{equation}
\end{definition}

\begin{lemma}
	For any fixed value of $r\geq 1$, the polynomial $f(r,x)$ has exactly one
	positive real root.
\end{lemma}
\begin{proof}
	We can check this directly for $r=1$. When $r>1$, the first two coefficients are positive and the last two
	coefficients are negative, so by Descartes' Rule of Signs, $f(r,x)$ has
	at most one positive real root. Since $f(r,0)<0$ and $\lim_{x
		\rightarrow \infty} f(r,x) > 0$, the polynomial $f(r,x)$ has
	exactly one positive real root.
\end{proof}

\begin{definition}
	We define $\alpha(r)$ to be the positive real root of $f(r,x)$.
\end{definition}

\begin{proposition} \label{worst 1ps for cusps}
	Let $r$ be a positive integer. Let $j = \lceil \alpha(r) \rceil$. 
	Then for all $N \geq j+2$, the persistent corner set for
	a cusp of order $r$ and the simplified problem is $I^{\simp} = \{j,j+1\}$.
\end{proposition}

In the proof, we will exhibit explicit formulas for a nonnegative solution $x$ to the KKT matrix
equation for this face. This will prove the claim.

First, we give a hint about how we obtained these quantities.

\subsection{How we obtained some formulas appearing in the
	proof} \label{subsec: how we obtained some formulas}

When $I = \{j,j+1\}$,  the piecewise linear graph through the points $\{
(\gamma_i,w_i)\}$ consists of three line segments. Let $y=m_k x +b_k$ be the equations
of these three line segments for $k=1,2,3$.

The KKT matrix equation variables $x_j$ and
$x_{j+1}$ represent parameters of $\face_{\gamma(I)}$, and we have
\begin{align}
	x_{j} & = -m_1 + m_2  \label{xj in terms of slopes} \\
	x_{j+1} &= -m_2 + m_3. \label{xjp1 in terms of slopes}
\end{align}

For the Simplified Problem, we have $m_3 = 0$ and $b_3 = 2$.

The middle line segment joins $(\gamma_{j},w_{j})$ and
$(\gamma_{j+1},w_{j+1})$. Since $j > \cond(\Gamma)$, we have
$\gamma_{j+1}-\gamma_j = 1$, and hence
\[
m_2 = (w_{j+1} -
w_{j})/(\gamma_{j+1}-\gamma_j)= (2-w_j)/1 = 2-w_j.
\]

We can use the equation for the first line segment to compute
$w_j$. This yields $w_j = m_1(j+r) + b_1$, and hence $m_2 =
2-m_1(j+r)-b_1$.

Substituting these expressions for $m_2$ and $m_3$ into (\ref{xj in terms
	of slopes}) and \ref{xjp1 in terms of slopes}) yields
\begin{align}
	x_{j} & = 2-m_1(j+r+1) -b_1 \label{eqn: xj in terms of m1 and b1} \\
	x_{j+1} &= -2 + m_1(j+r) +b_1. \label{eqn: xjpq in terms of m1 and b1} 
\end{align}

The simple structure of the corner sets $I = \{j,j+1\}$ helps us in
another way. Whenever we have two
consecutive corners, the optimisation problem breaks up: we can
optimize $\sum_{i=0}^{j} (w_i-a_i)^2$ and $\sum_{i=j+1}^{N}
(w_i-a_i)^2$ separately. When there are exactly two corners, as in $I
= \{ j, j+1\}$, the solutions to each of these problems is given by least squares
regression. We thus obtain formulas for $m_1$ and $b_1$ as follows.

\begin{lemma} \label{lem: m1 and b1}
	Let $\Gamma = \langle 2,r\rangle$. 
	\begin{enumerate}
		\item Applying least squares linear regression to $(\gamma_0,a_0),\ldots,(\gamma_j,a_j)$ yields the following formulas.
		\begin{align*}
			n &= j+1\\
			\sum \gamma_i w_i &= (r+j)(r+j+1) \\
			\sum \gamma_i &= \frac{1}{2}(j^2+2jr+j-r^2+r) \\
			\sum w_i &= 2j+2r+1\\
			\sum \gamma_i^2 &= -r^3+r^2j+rj^2+\frac{1}{3}j^3+\frac{1}{2}r^2+rj+\frac{1}{2}j^2+\frac{1}{2}r+\frac{1}{6}j
		\end{align*}
		\item Recall that 
		\begin{align*}
			m_1 &= \frac{n \left(\sum \gamma_i w_i\right) - \left(\sum \gamma_i\right) \left(\sum w_i\right)}{n\left(\sum \gamma_i^2\right) - \left(\sum \gamma_i\right)^2}\\
			b_1 &= \frac{1}{n}\left(\sum w_i\right) - m_1 \frac{1}{n}\left(\sum \gamma_i\right)
		\end{align*}
		\item The denominator of $m_1$ is
		\begin{align*}
			n\left(\sum \gamma_i^2\right) - \left(\sum \gamma_i\right)^2
			= \frac{1}{12}\left(j^4+4 j^3+(6 r^2+6 r+5) j^2+(12 r^2+12 r+2) j-3 r^4-6 r^3+3 r^2+6 r\right)
		\end{align*}
		We scale this polynomial by 12 so that the $j^4$ term is monic, and call this polynomial $h$. Then 
		\begin{align*}
			m_1 &= \frac{6((-2r+1)j^2+j+2r^3+r^2+r)}{h} \\
			b_1 &= \frac{2(j^4+(4 r+2) j^3+(12 r^2+6 r+2) j^2+(12 r^2+8 r+1) j-9 r^4+6 r^2+3 r)}{h}
		\end{align*}
	\end{enumerate}
\end{lemma}

Substituting the formulas for $m_1$ and $b_1$ from  Lemma \ref{lem: m1
	and b1} into the expression (\ref{eqn: xj in terms of m1 and b1}) yields
\begin{equation}
	x_{j}  = \frac{f}{h},
\end{equation}
where 
\begin{align*}
	f(r,j) &= (4 r-2) j^3+(6 r-6) j^2+(-12 r^3-6 r^2-4 r-4) j-30 r^3-18 r^2\\
	h(r,j) &= j^4+4 j^3+(6 r^2+6 r+5) j^2+(12 r^2+12 r+2) j-3 r^4-6 r^3+3 r^2+6 r
\end{align*}

We also find that
\[
x_{j+1} = -\frac{f(r,j-1)}{h}.
\]

\subsection{Proof of the main result for cusps}

\begin{proof}[Proof of Proposition \ref{worst 1ps for cusps}]
	
	We exhibit explicit formulas for a nonnegative solution $x$ to the KKT matrix
	equation for this face. This proves the claim.
	
	We begin by defining the following quantities. See Section
	\ref{subsec: how we obtained some formulas} for a discussion of how we
	obtained these formulas.
	
	\begin{align}
		h(r,j) &= j^4+4 j^3+(6 r^2+6 r+5) j^2+(12 r^2+12 r+2) j-3 r^4-6
		r^3+3 r^2+6 r \\
		x_{j}  & = \frac{f}{h}\\
		x_{j+1} &= -\frac{f(r,j-1)}{h} \\
		m_1 &= -x_j - x_{j+1}\\
		b_1 &= (j+r) (x_j+x_{j+1})+2+x_{j+1}      
	\end{align}

	We then define $x_k$ for $1 \leq k \leq j-1$ in terms of the
	quantities above.
	
	\begin{equation} \label{eqn: xk def}
		x_k = \left\{
		\begin{array}{ll}
			\mbox{} \rule{0pt}{14pt} \frac{1}{3}(k-1)k(k+1)m_1 + \frac{1}{2}k(k+1)b_1- 2k^2  & \text{ if } 1 \leq k \leq r-1 \\
			\mbox{}  \rule{0pt}{14pt} \frac{1}{6} (j-r) (j-r+1) ( (j-r-1) x_j + (j-r+2) x_{j+1}) & \text{ if } k=r \\
			\mbox{}  \rule{0pt}{14pt} \frac{1}{3} (j-k) (j-k+1) ( (j-k-1) x_j + (j-k+2) x_{j+1}) & \text{ if } r+1 \leq k \leq j-1 \\
			\mbox{}  \rule{0pt}{14pt} 0 & \text{ if } j+2 \leq k \leq N \\
			\mbox{}  \rule{0pt}{14pt} 2 & \text{ if } k=N+1 \\
		\end{array}
		\right.
	\end{equation}
	
	To finish the proof, we need to show that the following three conditions are satisfied.
	
	\begin{enumerate}
		\item[(i).] $x$ satisfies the KKT matrix equation for the Simplified Problem
		\item[(ii).] $x_i > 0$ for $i \in I$
		\item[(iii).] $x_i \geq 0$ for $i \not\in I$ and $i\leq N-1$. 
	\end{enumerate}

	To establish $(i)$, we need to verify several identities of rational functions.
	
	Recall the formulas for the KKT matrix equation:
	
	\[
	A_{i,j} = \left\{
	\begin{array}{ll}
		\gamma_j - \gamma_{j+1} & \text{if $j\leq N-1$, $j \not\in I$, and $i=j$} \\
		\gamma_{j+1} - \gamma_{j-1} & \text{if $j\leq N-1$, $j \not\in I$, and $i=j+1$} \\
		\gamma_{j-1}- \gamma_{j} & \text{if $j\leq N-1$, $j \not\in I$, and $i=j+2$} \\
		2(\gamma_j - \gamma_{i-1}) & \text{if $j\leq N-1$, $j \in I$, and $i \leq j$} \\
		2(\gamma_N - \gamma_{i-1}) & \text{if $j= N$} \\
		2 & \text{if $j= N+1$} \\
		0 & \text{otherwise}
	\end{array}
	\right.
	\]
	The vector $a$ on the right hand side of the KKT matrix equation is given by the following formula.
	\[
	a_i = \left\{
	\begin{array}{ll}
		\gamma_1 & \text{if $i=1$} \\
		\gamma_{i} - \gamma_{i-2} & \text{if $2 \leq i \leq N$} \\
		2 & \text{if $i=N+1$ (Simplified Problem)} \\
	\end{array}
	\right.
	\]
	
	For the semigroup $\Gamma = \langle 2, 2r+1 \rangle$, we have
	\[
	\gamma_i = \left\{
	\begin{array}{ll}
		2i & \text{if $0 \leq i \leq r$} \\
		i+r  & \text{if $r+1 \leq i \leq N$} 
	\end{array}
	\right.
	\]
	Thus
	\[
	a_i = \left\{
	\begin{array}{ll}
		2 & \text{if $i=1$} \\
		4 & \text{if $2 \leq i \leq r$} \\
		3 & \text{if $i=r+1$} \\
		2 & \text{if $r+2 \leq i$} 
	\end{array}
	\right.
	\]
	
	Now we check each row of the KKT matrix equation.    
	
	In row 1, we have $A_{1,k} = 0$ when $k \not\in \{ 1,j,j+1,N,N+1\}$, and $x_N = 0$. Then
	\begin{align*}
		\lefteqn{A_{1,1} x_1 + A_{1,j} x_j + A_{1,j+1} x_{j+1} +A_{1,N+1} x_{N+1} } \\
		&= (-2) (b_1-2) + (2(j+r)) x_j + (2(j+r+1))x_{j+1} + (2)( 2) \\
		&= 4 \\
		&= 2 a_1.
	\end{align*}
	
	In row 2, we have $A_{2,k} = 0$ when $k \not\in \{ 1,2,j,j+1,N,N+1\}$, and $x_N = 0$. Then
	\begin{align*}
		\lefteqn{A_{2,1} x_1 + A_{2,2} x_2+ A_{2,j} x_j + A_{2,j+1} x_{j+1} +A_{2,N+1} x_{N+1} } \\
		&= (4) (b_1-2)+ (-2) (2m_1+3b_1-8)+ (2(j+r-2)) x_j + (2(j+r-1))x_{j+1} + (2)( 2) \\
		&= 8 \\
		&= 2 a_2.
	\end{align*}
	
	Let $3 \leq i \leq r-1$. In row $i$, we have $A_{i,k} = 0$ when $k \not\in \{ i-2,i-1,i,j,j+1,N,N+1\}$, and $x_N = 0$. Then
	\begin{align*}
		\lefteqn{A_{i,i-2} x_{i-2} + A_{i,i-1} x_{i-1} + A_{i,i} x_{i} + A_{i,j} x_j + A_{i,j+1} x_{j+1} +A_{i,N+1} x_{N+1} } \\
		&= (-2) (\frac{1}{3}(i-3)(i-2)(i-1) m_1 + \frac{1}{2}(i-2)(i-1) b_1 -2(i-2)^2) \\
		&\mbox{} \qquad +(4) (\frac{1}{3}(i-2)(i-1)i m_1 + \frac{1}{2}(i-1)(i) b_1 -2(i-1)^2) \\
		&\mbox{} \qquad +(-2) (\frac{1}{3}(i-1)i(i+1) m_1 + \frac{1}{2}(i)(i+1) b_1 -2i^2)\\
		& \mbox{} \qquad+ (2(j+r-2(i-1))) x_j + (2(j+r+1-2(i-1)))x_{j+1} + (2)( 2) \\
		&= 8 \\
		&= 2 a_i.
	\end{align*}

	In row $r$, we have $A_{r,k} = 0$ when $k \not\in \{ r-2,r-1,r,j,j+1,N,N+1\}$, and $x_N = 0$. Then
	\begin{align*}
		\lefteqn{A_{r,r-2} x_{r-2} + A_{r,r-1} x_{r-1} + A_{r,r} x_{r} + A_{r,j} x_j + A_{r,j+1} x_{j+1} +A_{r,N+1} x_{N+1} } \\
		&= (-2) (\frac{1}{3}(r-3)(r-2)(r-1) m_1 + \frac{1}{2}(r-2)(r-1) b_1 -2(r-2)^2) \\
		&\mbox{} \qquad +(4) (\frac{1}{3}(r-2)(r-1)r m_1 + \frac{1}{2}(r-1)(r) b_1 -2(r-1)^2) \\
		&\mbox{} \qquad+ (-1) (  \frac{1}{6} (j-r) (j-r+1) ( (j-r-1) x_j + (j-r+2) x_{j+1}) )\\
		& \mbox{} \qquad+ (2(j-r+2)) x_j + (j-r+3)x_{j+1} + (2)( 2) \\
		&= 8 \\
		&= 2 a_r.
	\end{align*}

	
	In row $r+1$, we have $A_{r+1,k} = 0$ when $k \not\in \{ r-1,r,r+1,j,j+1,N,N+1\}$, and $x_N = 0$. Then
	\begin{align*}
		\lefteqn{A_{r+1,r-1} x_{r-1} + A_{r+1,r} x_{r} + A_{r+1,r+1} x_{r+1} + A_{r+1,j} x_j + A_{r+1,j+1} x_{j+1} +A_{r+1,N+1} x_{N+1} } \\
		&= (-2) (\frac{1}{3}(r-2)(r-1)r m_1 + \frac{1}{2}(r-1)(r) b_1 -2(r-1)^2) \\
		&\mbox{} \qquad +(3)   \frac{1}{6} (j-r) (j-r+1) ( (j-r-1) x_j + (j-r+2) x_{j+1}) ) \\
		&\mbox{} \qquad+ (-1) ( \frac{1}{3} (j-(r+1)) (j-(r+1)+1) ( (j-(r+1)-1) x_j + (j-(r+1)+2) x_{j+1}) )\\
		& \mbox{} \qquad+ (2(j-r)) x_j + (2(j-r+1))x_{j+1} + (2)( 2) \\
		&= 6 \\
		&= 2 a_{r+1}.
	\end{align*}

	In row $r+2$, we have $A_{r+2,k} = 0$ when $k \not\in \{r,r+1,r+2,j,j+1,N,N+1\}$, and $x_N = 0$. Then
	\begin{align*}
		\lefteqn{A_{r+2,r} x_{r} + A_{r+2,r+1} x_{r+1} + A_{r+2,r+2} x_{r+2} + A_{r+2,j} x_j + A_{r+2,j+1} x_{j+1} +A_{r+2,N+1} x_{N+1} } \\
		&= (-2) \frac{1}{6} (j-r) (j-r+1) ( (j-r-1) x_j + (j-r+2) x_{j+1}) ) \\
		&\mbox{} \qquad +(2)  ( \frac{1}{3} (j-(r+1)) (j-(r+1)+1) ( (j-(r+1)-1) x_j + (j-(r+1)+2) x_{j+1}) )\\
		&\mbox{} \qquad+ (-1) ( \frac{1}{3} (j-(r+2)) (j-(r+2)+1) ( (j-(r+2)-1) x_j + (j-(r+2)+2) x_{j+1}) )\\
		& \mbox{} \qquad+ (2(j-r-1)) x_j + (2(j-r))x_{j+1} + (2)( 2) \\
		&= 4 \\
		&= 2 a_{r+2}.
	\end{align*}
	
	Let $r+3 \leq i \leq j-1$. In row $i$, we have $A_{i,k} = 0$ when $k \not\in \{i-2,i-1,i,j,j+1,N,N+1\}$, and $x_N = 0$. Then
	\begin{align*}
		\lefteqn{A_{i,i-2} x_{i-2} + A_{i,i-1} x_{i-1} + A_{i,i} x_{i} + A_{i,j} x_j + A_{i,j+1} x_{j+1} +A_{i,N+1} x_{N+1} } \\
		&= (-1) (\frac{1}{3} (j-(i-2)) (j-(i-2)+1) ( (j-(i-2)-1) x_j + (j-(i-2)+2) x_{j+1}))\\
		&\mbox{} \qquad +(2)  (\frac{1}{3} (j-(i-1)) (j-(i-1)+1) ( (j-(i-1)-1) x_j + (j-(i-1)+2) x_{j+1}))\\
		&\mbox{} \qquad +(-1)  (\frac{1}{3} (j-k) (j-i+1) ( (j-i-1) x_j + (j-i+2) x_{j+1}))\\
		& \mbox{} \qquad+ (2(j-i+1)) x_j + (2(j-i+2))x_{j+1} + (2)( 2)  \\
		&= 4 \\
		&= 2 a_i.
	\end{align*}
	
	In row $j$, we have $A_{j,k} = 0$ when $k \not\in \{j-2,j-1,j,j+1,N,N+1\}$, and $x_N = 0$. Then
	\begin{align*}
		\lefteqn{A_{j,j-2} x_{j-2} + A_{j,j-1} x_{j-1} + A_{j,j} x_j + A_{j,j+1} x_{j+1} +A_{j,N+1} x_{N+1} } \\
		&= (-1) (2 ( x_j + 4 x_{j+1}))+(2)  (2 x_{j+1})+ (2) x_j + (4)x_{j+1} + (2)( 2)  \\
		&= 4 \\
		&= 2 a_j.
	\end{align*}

	In row $j+1$, we have $A_{j+1,k} = 0$ when $k \not\in \{j-1,j+1,N,N+1\}$, and $x_N = 0$. Then
	\begin{align*}
		\lefteqn{A_{j+1,j-1} x_{j-1}  + A_{j+1,j+1} x_{j+1} +A_{j+1,N+1} x_{N+1} } \\
		&= (-1) (2 x_{j+1})+(2)  (x_{j+1})+ (2)( 2)  \\
		&= 4 \\
		&= 2 a_{j+1}.
	\end{align*}

	Let $j+2 \leq i \leq N+1$. We have   $A_{i,k} = 0$ when $k \not\in \{i-2,i-1,i,N,N+1\}$. But $x_{i-2}=x_{i-1}=x_{i}=x_N=0$, so this row of the KKT matrix equation encodes the equation $A_{i,N+1} x_{N+1} = 2 \cdot 2 = 2 a_i$.
	
	Thus, $x$ is a solution to the KKT matrix equation for the Simplified Problem for this face.

	Next, we establish $(ii)$. We have $I = \{j,j+1\}$, so we need to analyze $x_j = f(r,j)/h$ and $x_{j+1}= -f(r,j-1)/h $. First, we argue that the denominator is positive. For this, fix $r>0$, and consider $h(r,j)$ as a polynomial in $j$. We have $\frac{dh}{dj}>0$ and $h(r,r) >0$, so $h(r,j)>0$ for all $j>r$.

	Now consider the numerators. The hypothesis that $j = \lceil \alpha(r) \rceil$ is equivalent to the statement that $j$ is the smallest integer value of $x$ for which $f(r,x)$ is positive. Equivalently, $j$ is the unique positive integer such that $f(r,j-1) < 0$ and $f(r,j)>0$.  It follows that $x_j > 0$ and $x_{j+1} > 0$.
	
	Next, we establish $(iii)$. For $r \leq k \leq j-1$, the expressions used to define $x_k$ are nonnegative linear combinations of $x_j$ and $x_{j+1}$, so $x_k \geq 0$ when $k$ is in this range.
	
	To establish the result when $k \leq r-1$, we define an auxiliary sequence $(\overline{x})$ as follows.
	
	Define
	\[
	\overline{x}_k := 
	\frac{x_k}{\frac{k(k+1)}{2}} 
	\]
	for $k \leq r-1$.

	Combining these definitions with the formulas in (\ref{eqn: xk def}) yields
	\[
	\overline{x}_k = \left\{ \begin{array}{ll}
		\frac{2(k-1)}{3} m_1 + b_1 - \frac{4k}{k+1} & \text{ if } 1 \leq k \leq r-1 \\
		\frac{2(r-1)}{3} m_1  + b_1 - \frac{4r}{r+1} & \text{ if } k=r 
	\end{array} \right.
	\]
	As $k$ increases, the fractions $\frac{2(k-1)}{3}$ and $\frac{4k}{k+1}$ increase. But these terms appear in the formula for $\overline{x}_k$ with negative coefficients (recall that $m_1<0$). So the sequence $\overline{x}_k $ decreases as $k$ increases. The denominators used to define $\overline{x}_k $ from $x_k$ are positive. Since $x_r$ is positive, $\overline{x}_r$ is positive, so $\overline{x}_1,\ldots,\overline{x}_{r-1}$ are positive, and hence $x_1,\ldots,x_{r-1}$ are positive.

\end{proof}

\section*{References}
\bibliographystyle{amsplain}

\end{document}